\documentclass[letterpaper,12pt,peerreviewca,draftcls]{IEEEtran}
\usepackage{csm16}
\usepackage[nolists,nomarkers,tablesfirst,figuresonly]{endfloat} 
\usepackage[utf8]{inputenc}
\usepackage{amsmath}
\usepackage{amssymb}
\usepackage{color}
\usepackage{graphicx}
\usepackage{BernsteinStyle2}
\usepackage{wrapfig}

\newcommand{\shiftq}{{\textbf{\textrm{q}}}}
\newcommand{\SeqPhi}{(\phi_k)_{k=0}^\infty}
\newcommand{\Zplus}{\{0,1, \ldots \}}

\newcommand{\EndExample}{$\hfill\mbox{\Large$\diamond$}$}

\usepackage{booktabs}
\newcommand{\tabitem}{~~\llap{\textbullet}~~}

\renewcommand{\QED}{$\hfill\mbox{\large$\square$}$}

\usepackage{subcaption}
\usepackage[noadjust]{cite}

\usepackage[left,pagewise]{lineno} 

\makeatletter
\let\NAT@parse\undefined
\makeatother

\usepackage{hyperref}
\usepackage{xcolor}
\hypersetup{
    colorlinks,
    linkcolor={blue!100!black},
    citecolor={blue!50!black},
    urlcolor={blue!80!black}
}

\title{\LARGE Recursive Least Squares \\ with Variable-Direction Forgetting \\
\large Compensating for the loss of persistency}

\author{Ankit Goel, Adam L. Bruce, and Dennis S. Bernstein\\POC: A.\ Goel (ankgoel@umich.edu)}
 
\begin{document}
\maketitle
\CSMsetup
\linenumbers \modulolinenumbers[2] 

The ability to estimate parameters depends on two things, namely, identifiability \cite{grewal1976identifiability}, which is ability to distinguish distinct parameters, and persistent excitation, which refers to the spectral content of the signals needed to ensure convergence of the parameter estimates to the true parameter values \cite{Mareels1986,Mareels1987,Mareels1988}.
Roughly speaking, the level of persistency must be commensurate with the number of unknown parameters.
For example, a harmonic input has two-dimensional persistency and thus can be used to identify two parameters, whereas white noise is sufficiently persistent for identifying an arbitrary number of parameters.
Within the context of adaptive control, persistent excitation is needed to avoid bursting  \cite{anderson1985adaptive}; 
recent research has focused on relaxing these requirements \cite{Chowdhary2010,chowdhary2014exponential,aranovskiy2017performance}.

Under persistent excitation, a key issue in practice is the rate of convergence, especially under changing conditions.
For example, the parameters of a system may change abruptly, and the goal is to ensure fast convergence to the modified parameter values.
In this case, it turns out that the rate of convergence depends on the ability to forget past parameters and incorporate new information.
As discussed in ``Summary,''  the ability to accommodate new information depends on the ability to forget;
the ability to forget is thus crucial to the ability to learn.
This paradox is widely recognized, and effective forgetting is of intense interest in machine learning \cite{panda2018asp,allred2016forcedfiring,fremaux2016neuromodulation,han2017deepspikingenergy}.

In the first half of the present article, classical forgetting within the context of recursive least squares (RLS) is considered. 
%
In the classical RLS formulation \cite{AseemRLS,ljung,sayed,astromCCS}, a constant forgetting factor $\lambda\in(0,1]$ can be set by the user.
However, it often occurs in practice that the performance of RLS is extremely sensitive to the choice of $\lambda$, and suitable values in the range $0.99$ to $0.9999$ are typically found by trial-and-error testing.
This difficulty has motivated extensions of classical RLS in the form of variable-rate forgetting \cite{FortescueKershenbaumYdstie, PaleologuBenestySilviu, LeungAndSo,SongLimBaekSung,ParkJunKim,AliHoaggMossbergBernstein,canetti1989convergence}, constant trace adjustment, covariance resetting, and covariance modification \cite{efra,goodwin1983deterministic}.  

In the second half of this article, {\it variable-direction forgetting} (VDF), a technique that complements variable-rate forgetting is considered. 
%
Direction-dependent forgetting has been widely studied within the context of recursive least squares \cite{kulhavy1987restricted,kulhavy1984tracking, kreisselmeier1990stabilized,Cao2000,kubin1988stabilization,bittanti1990convergence,bittanti1990exponential}.
In the absence of persistent excitation, new information is confined to a limited number of directions.
The goal of VDF is thus to determine these directions and thereby constrain forgetting to the directions in which new information is available.
VDF allows RLS to operate without divergence during periods of loss of persistency.

%

%
The goal of this tutorial article is to investigate the effect of forgetting within the context of RLS in order to motivate the need for VDF.  
With this motivation in mind, the article develops and illustrates RLS with VDF.
The presentation is intended for graduate students who may wish to understand and apply this technique to system identification for modeling and adaptive control.
Table \ref{tab:Results} and \ref{tab:Examples} summarizes the results and examples in this article. 
Some of the content in this article appeared in preliminary form in \cite{goel2018targeted}.

Although, in practical applications, all sensor measurements are corrupted by noise,
the effect of sensor noise is not considered in this article in order to focus on the loss of  persistency. 
Alternative interpretations of RLS in the special case of zero-mean, white sensor noise are presented in
``RLS as a One-Step Optimal Predictor'' and
``RLS as a Maximum Likelihood Estimator''.

\begin{table}[h!]
    \caption{Summary of definitions and results in this article.}
    \centering
    \begin{tabular}{|l|l|}
        \hline
        Definition \ref{def:persistent_Exc} & Persistently exciting regressor 
        \\ \hline
        Definition \ref{def:LS1} & Lyapunov stable equilibrium
        \\ \hline
        Definition \ref{def:ULS1} & Uniformly Lyapunov stable equilibrium
        \\ \hline
        Definition \ref{def:GAS1} & Globally aymptotically stable equilibrium
        \\ \hline
        Definition \ref{def:GUGS1} & Uniformly globally geometrically stable equilibrium
        \\ \hline
        %
        %
        %
        %
        %
        Theorem \ref{theorem_RLS}-\ref{theorem_RLS_reverse} &  Recursive least squares (RLS)
        \\ \hline
        Theorem \ref{theo:LS1}-\ref{theo:GUGS1} & Lyapunov stability theorems
        \\ \hline
        Theorem \ref{prop:RLS_stability_LP} & Lyapunov analysis of RLS for $\lambda \in (0,1)$
        \\ \hline
        Theorem \ref{prop:AS_GS} & Stability analysis of RLS for $\lambda \in (0,1]$ based on $\theta_k$
        \\ \hline
        Theorem \ref{theorem_RLS_VDF} & A Quadratic Cost Function for Variable-Direction RLS
        \\ \hline
        Proposition \ref{prop:Pk_MIL} & Recursive update of $P_{k}\inv$ with uniform-direction forgetting
        \\ \hline
        Proposition \ref{prop:theta_in_RangePHI} & Data-dependent subspace constraint on $\theta_k$
        \\ \hline
        Proposition \ref{prop:Pkinv_bounds_wo_forgetting} & Bounds on $P_k$ for $\lambda = 1$ 
        \\ \hline
        Proposition \ref{prop:Pkinv_bounds} & Bounds on $P_k$ for $\lambda \in (0,1)$
        \\ \hline
        Proposition \ref{prop:Pkinv_boundsconv} & Converse of Proposition \ref{prop:Pkinv_bounds}
        \\ \hline
        Proposition \ref{prop:z_converges} & Convergence of $z_k$ with uniform-direction forgetting
        \\ \hline
        Proposition \ref{prop:Ak_PE} & Persistent excitation and $\SA_k$
        \\ \hline
        Proposition \ref{prop:Pk_MIL_VDF} & Recursive update of $P_{k}\inv$ with variable-direction forgetting
        \\ \hline
        Proposition \ref{prop:z_converges_VDF} & Convergence of $z_k$ with variable-direction forgetting
        \\ \hline
        Proposition \ref{prop:Pkinv_bounds_VDF} & Bounds on $P_k$ with variable-direction forgetting
        \\ \hline
    \end{tabular}
    \label{tab:Results}
\end{table}

\begin{table}[h!]
    \caption{Summary of examples in this article.}
    \centering
    \begin{tabular}{|l|l|}
        \hline
        Example \ref{exam:CounterExample} & $P_k$ converges to zero without persistent excitation
        \\ \hline
        Example \ref{exam:PE_Pk_bounds_1} & Persistent excitation and bounds on $P_k^{-1}$
        \\ \hline
        Example \ref{exam:PE_Pk_bounds_2} & Lack of persistent excitation and bounds on $P_k^{-1}$
        \\ \hline
        Example \ref{exam:LoP_zto0_theta_conv} & Convergence of $z_k$ and $\theta_k$
        \\ \hline
        Example \ref{exam:persistency_conditionNumber} & Using $\kappa(P_k)$ to determine whether $\SeqPhi$ is persistently exciting
        \\ \hline
        Example \ref{exam:ConvRate_lambda} & Effect of $\lambda$ on the rate of convergence of $\theta_k$
        \\ \hline
        Example \ref{exam:ScalarEstimation} & Lack of persistent excitation in scalar estimation
        \\ \hline
        Example \ref{exam:LoP_n2} & Subspace constrained regressor
        \\ \hline
        Example \ref{exam:LackofPE_theta_k} & Effect of lack of persistent excitation on $\theta_k$
        \\ \hline
        Example \ref{exam:Persistency_InformationContent} & Lack of persistent excitation and the information-rich subspace  
        \\ \hline
        Example \ref{exam:Persistency_InformationContent_with_TF} & 
        Variable-direction forgetting for a regressor lacking persistent excitation
        \\ \hline
        Example \ref{exam:TF_effect_on_theta_k} & Effect of variable-direction forgetting on $\theta_k$
        \\ \hline
    \end{tabular}
    \label{tab:Examples}
\end{table}

\section{Recursive Least Squares}

Consider the model 
\begin{align}
    y_k 
        =
            \phi_k \theta,
    \label{eq:process}
\end{align}
where, for all $k\ge0,$ $y_k \in \BBR^p$ is the measurement, 
$\phi_k \in \BBR^{p \times n }$ is the regressor matrix, and
$\theta \in \BBR^n$ is the vector of unknown parameters. 
The goal is to estimate $\theta$ as new data become available.
One approach to this problem is to minimize the quadratic cost function
\begin{align} 
    J_k({\hat\theta})
        &\isdef
            \sum_{i=0}^k
            \lambda^{k-i} (y_{i} - \phi_{i} {\hat\theta})^\rmT
            (y_{i} - \phi_{i} {\hat\theta}) 
            + 
            \lambda^{k+1}  ({\hat\theta}-\theta_0) ^\rmT
            R
            ({\hat\theta}-\theta_0),
    \label{eq:J_LS}
\end{align}
where $\lambda \in (0,1]$ is the {\it forgetting factor}, $R \in \BBR^{n \times n}$ is positive definite, and $\theta_0\in\BBR^n$ is the initial estimate of $\theta$. %
The forgetting factor applies higher weighting to more recent data, thereby enhancing the ability of RLS to use incoming data to estimate time-varying parameters. 
The following result is {\it recursive least squares}.

\begin{theo}\label{theorem_RLS}
    For all $k\ge0$, let $\phi_k \in \BBR^{p \times n}$ and $y_k \in \BBR^p$, let $R\in\BBR^{n\times n}$ be positive definite, and define 
    $P_{0} \isdef R^{-1}$, $\theta_0 \in \BBR^n$, and  $\lambda \in (0,1]$.
Furthermore, for all $k \ge 0$, denote the minimizer of \eqref{eq:J_LS} by
\begin{align}
    \theta_{k+1}
        =
            \underset{ \hat\theta \in \BBR^n  }{\operatorname{argmin}} \
            J_k({\hat\theta}).
    \label{eq:theta_minimizer_def}
\end{align}
Then, for all $k \ge 0$, $\theta_{k+1}$ is given by 
    \begin{align}
        P_{k+1}
            &=
                \dfrac{1}{\lambda}P_{k} -
                \dfrac{1}{\lambda}
                    P_{k} \phi_k ^\rmT
                    \left(
                        \lambda I_p + \phi_k  P_{k} \phi_k^\rmT
                    \right)^{-1}
                    \phi_k P_{k}
                ,
        \label{eq:PUpdate} \\
        \theta_{k+1}
            &=
                \theta_{k} +
                P_{k+1} \phi_k^\rmT (y_k - \phi_k \theta_{k}).
        \label{eq:thetaUpdate}
    \end{align}    
\end{theo}

\begin{proof}
    See \cite{AseemRLS}.
\end{proof}

The following result is a variation of Theorem \ref{theorem_RLS}, where the updates of $P_k$ and $\theta_k$ are reversed.
\begin{theo}
    \label{theorem_RLS_reverse}
    For all $k\ge0$, let $\phi_k \in \BBR^{p \times n}$ and $y_k \in \BBR^p$, let $R\in\BBR^{n\times n}$ be positive definite, and define 
    $P_{0} \isdef R^{-1}$, $\theta_0 \in \BBR^n$, and  $\lambda \in (0,1]$.
Furthermore, for all $k \ge 0$, denote the minimizer of \eqref{eq:J_LS} by \eqref{eq:theta_minimizer_def}.
Then, for all $k \ge 0$, $\theta_{k+1}$ is given by 
\begin{align}
    \theta_{k+1} 
        &=
            \theta_k  +  P_{k}\phi_k^\rmT (\lambda I + \phi_kP_{k}\phi_k^\rmT )^{-1} (  y_k - \phi_k \theta_k ), \label{eq:theta_update_WithInverse}\\
    P_{k+1} 
        &=
            \frac{ 1}{\lambda }P_{k} -
            \frac{ 1}{\lambda }P_{k}\phi_k^\rmT  
            ( \lambda I +  \phi_kP_{k}\phi_k^\rmT )^{-1} \phi_k P_{k}.
            \label{eq:P_update_WithInverse}
\end{align}  
\end{theo}

\begin{proof}
    See \cite{AseemRLS}.
\end{proof}

\begin{prop}
    \label{prop:Pk_MIL}
    Let $\lambda \in (0,\infty)$, and let $(P_k)_{k=0}^\infty$ be a sequence of $n\times n$ positive-definite matrices.
    Then, for all $k\ge0,$ $(P_k)_{k=0}^\infty$ satisfies \eqref{eq:PUpdate} if and only if, for all $k\ge0,$
    $(P_k)_{k=0}^\infty$ satisfies
    \begin{align}
        P_{k+1}^{-1}
            =
                \lambda P_{k}^{-1}
                + 
                \phi_{k}^\rmT \phi_{k}.
        \label{eq:Pk_recursive}
    \end{align}
\end{prop}
\begin{proof}
    To prove necessity, it follows from \eqref{eq:Pk_recursive} and matrix-inversion lemma, that
    \begin{align}
        P_{k+1}
            &=
                (
                \lambda P_{k}\inv
                + 
                \phi_{k}^\rmT \phi_{k}
                )\inv
            \nn \\
            &=
                (\lambda P_{k}\inv)\inv-
                (\lambda P_{k}\inv)\inv \phi_k^\rmT
                (I_p + \phi_k (\lambda P_{k}\inv)\inv \phi_k^\rmT )\inv
                \phi_k (\lambda P_{k}\inv)\inv
            \nn \\
            &=
                \dfrac{1}{\lambda}P_{k} -
                \dfrac{1}{\lambda}
                    P_{k} \phi_k ^\rmT
                    \left(
                        \lambda I_p + \phi_k  P_{k} \phi_k^\rmT
                    \right)^{-1}
                    \phi_k P_{k} .
            \nn
    \end{align}
    Reversing these steps proves sufficiency.   
\end{proof}

Let $k\ge0.$    By defining the {\it parameter error} 
\begin{align}
    \tilde \theta_k
        \isdef 
            \theta_k - \theta,
    \label{eq:est_err}
\end{align}
it follows that 
\begin{align}
    \phi_i \theta_{k} - y_i = \phi_i\tilde\theta_{k}.\label{eqn8}
\end{align}
Using \eqref{eqn8} with $k$ replaced by $k+1$, it follows that the minimum value of $J_k$ is given by
\begin{align} 
    J_k({\theta_{k+1}})
        &=
            \sum_{i=0}^k
            \lambda^{k-i} 
            \tilde \theta_{k+1}^\rmT \phi_i^\rmT
            \phi_i \tilde \theta_{k+1}
            + 
            \lambda^{k+1}  
            (\tilde \theta_{k+1} - \tilde \theta_0)^\rmT
            R
            (\tilde \theta_{k+1} - \tilde \theta_0).
    \label{eq:J_LS_2}
\end{align}
Furthermore, \eqref{eq:thetaUpdate} and \eqref{eq:est_err} imply that $\tilde \theta_k$ satisfies
\begin{align}
    \tilde \theta_{k+1}
        &= 
            (I_n - P_{k+1} \phi_k^\rmT \phi_k) \tilde \theta_k 
        \label{eq:theta_error1} \\
        &=
            \lambda P_{k+1} P_k^{-1} \tilde \theta_k.
    \label{eq:theta_error2}
\end{align}
Finally, it follows from \eqref{eq:theta_error2} that, for all $k,l\ge0$,
\begin{align}
    \tilde \theta_{k}
        = 
             \lambda^{k-l} P_{k} P_{l}^{-1}  \tilde \theta_{l}.
    \label{eq:theta_error_explicit}
\end{align}


%
The following result shows that the estimate $\theta_k$ of $\theta$ is constrained to a data-dependent subspace. Let $\SR(A)$ denote the range of the matrix $A$.

\begin{prop}
    \label{prop:theta_in_RangePHI}
    For all $k\ge0$, let $\phi_k \in \BBR^{p \times n}$ and $y_k \in \BBR^p$, let $R\in\BBR^{n\times n}$ be positive definite, let $\theta_0 \in \BBR^n$, let  $\lambda \in (0,1]$, and define $\theta_{k+1}$ by \eqref{eq:theta_minimizer_def}.
    Then, $\theta_{k+1}$ satisfies
    \begin{align}
        \left( \sum_{i=0}^k \lambda^{k-i} \phi_i^\rmT \phi_i + \lambda^{k+1}R \right)
        \theta_{k+1}
            =
                \sum_{i=0}^k \lambda^{k-i} \phi_i^\rmT y_i +
                \lambda^{k+1} R \theta_0.
        \label{eq:Atheta_b}
    \end{align}
    Furthermore, 
    \begin{align}
        \theta_{k+1} 
            \in
                 \SR(\Phi_k^\rmT \Phi_k + R^{-1} \Phi_k^\rmT \Phi_k R^{-1} + \theta_0 \theta_0^\rmT),
        \label{eq:theta_in_range_PHI}
    \end{align}
    where
    \begin{align}
        \Phi_k 
            &\isdef
                [
                    \phi_0^\rmT \ \
                    \cdots \ \
                    \phi_k^\rmT
                ] ^\rmT
                \in \BBR^{(k+1)p \times n}.
    \end{align}
\end{prop}

\begin{proof}
    Note that
    \begin{align}
        J_k({\hat \theta} )
            =
                \hat \theta ^\rmT A_k \hat \theta + 
                \hat \theta ^\rmT b_k  + 
                c_k, \nn
    \end{align}
    where 
    \begin{align}
        A_k
            &\isdef
                \sum_{i=0}^k \lambda^{k-i} \phi_i^\rmT \phi_i + \lambda^{k+1}R, 
            \nn \\
        b_k
            &\isdef
                \sum_{i=0}^k -\lambda^{k-i} \phi_i^\rmT y_i -
                \lambda^{k+1} R \theta_0,
            \nn \\
        c_k
            &\isdef
                \sum_{i=0}^k \lambda^{k-i} y_i^\rmT y_i +
                \lambda^{k+1} \theta_0^\rmT R \theta_0.\nn
    \end{align}
    Since $A_k$ is positive definite, it follows from Lemma 1 in \cite{AseemRLS} that the minimizer $\theta_{k+1}$ of $J_k$ satisfies \eqref{eq:Atheta_b}.

    Next, define $W_k \isdef
                \ {\rm diag} (\lambda^{-1} I_p, \ldots, \lambda^{-1-k}I_p)
                \in \BBR^{(k+1)p \times (k+1)p}$.
    Using \eqref{eq:Atheta_b} and Lemma \ref{lemma:y_in_X} from ``Three Useful Lemmas,'' it follows that
    \begin{align}
        \theta_{k+1}
            &=
                \left( 
                    I_n  + \Phi_k^\rmT W_k \Phi_k
                \right)^{-1}
                \left(
                \sum_{i=0}^k
                \lambda^{-i-1} R^{-1} \phi_i^\rmT y_i  
                + 
                \theta_0 \right)
                \nn \\
            &=
                \sum_{i=0}^k
                    \left( 
                        I_n  + \Phi_k^\rmT W_k \Phi_k
                    \right)^{-1}
                    \lambda^{-i-1} R^{-1} \phi_i^\rmT y_i 
                + 
                \left( 
                        I_n  + \Phi_k^\rmT W_k \Phi_k
                \right)^{-1}
                \theta_0 
                \nn \\       
            &\in
                \sum_{i=0}^k
                \SR([\Phi_k^\rmT \ \ R^{-1} \phi_i^\rmT ])
                +
                \SR([\Phi_k^\rmT \ \ \theta_0 ]) \nn \\
            &=
                \SR([\Phi_k^\rmT \ \ R^{-1} \Phi_k^\rmT  \ \ \theta_0 ])
            \nn \\
            &=
                \SR(\Phi_k^\rmT \Phi_k + R^{-1} \Phi_k^\rmT \Phi_k R^{-1} + \theta_0 \theta_0^\rmT)
                .\nn
    \end{align}
\end{proof}

Table \ref{tab:RLS_expressions} summarizes various expressions for the RLS variables. 
\begin{table}[h!]
    \caption{Alternative expressions for the RLS variables.}
    \centering
    \begin{tabular}{|c|l|c|}
        \hline
        Variable &
        Expression & 
        Equation    \\
        \hline
        $P_k$  & 
        \tabitem
        $P_{k+1}
            =
                \dfrac{1}{\lambda}P_{k} -
                \dfrac{1}{\lambda}
                    P_{k} \phi_k ^\rmT
                    \left(
                        \lambda I_p + \phi_k  P_{k} \phi_k^\rmT
                    \right)^{-1}
                    \phi_k P_{k}$       
        &    
        \eqref{eq:PUpdate}
        \\
        
        &
        \tabitem $P_{k+1}^{-1}
        =
            \lambda P_{k}^{-1}
            + 
            \phi_{k}^\rmT \phi_{k}$
        &
        \eqref{eq:Pk_recursive}
        \\
        
        &
        \tabitem $P_{k+1}^{-1}
        =
            \lambda^{k+1} P_{0}^{-1}
            + 
            \sum_{i=0}^k \lambda^{k-i}
            \phi_{i}^\rmT \phi_{i}$
        &
        \eqref{eq:Pk_recursive}
        \\
        \hline
        $\theta_k$  &    
        \tabitem $\theta_{k+1}
            =
                \theta_{k} +
                P_{k+1} \phi_k^\rmT (y_k - \phi_k \theta_{k})$    
        &
        \eqref{eq:thetaUpdate}
        \\
         &
         \tabitem $\theta_{k+1}
            =
                \theta_{k} +
                P_{k} \phi_k^\rmT 
                (
                        \lambda I_p + \phi_k  P_{k} \phi_k^\rmT
                    )^{-1}
                (y_k - \phi_k \theta_{k})$
        &
        \eqref{eq:theta_update_WithInverse}
        \\
        &
        \tabitem
        $\theta_{k+1}
            =
                P_{k+1}
                \left(
                \sum_{i=0}^k \lambda^{k-i} \phi_i^\rmT y_i +
                \lambda^{k+1} P_0^{-1} \theta_0
                \right)$
        &
        \eqref{eq:Atheta_b}
        \\
        \hline
        $\tilde \theta_k$  &    
        \tabitem 
        $\tilde \theta_k = \theta_k - \theta$    
        &
        \eqref{eq:est_err}
        \\
        
        &
        \tabitem 
        $\tilde \theta_{k+1} = (I_n - P_{k+1} \phi_k^\rmT \phi_k) \tilde \theta_k $
        &
        \eqref{eq:theta_error1}
        \\
         
        &
        \tabitem 
        $\tilde \theta_{k+1} = \lambda P_{k+1} P_k^{-1} \tilde \theta_k $
        &
        \eqref{eq:theta_error2}
        \\
          
        &
        \tabitem 
        $\tilde \theta_k = \lambda^{k-l} P_{k} P_{l}^{-1}  \tilde \theta_{l} $
        &
        \eqref{eq:theta_error_explicit}
        \\
        \hline
    \end{tabular}
    \label{tab:RLS_expressions}
\end{table}


\section{Persistent Excitation and Forgetting}

This section defines persistent excitation of the regressor sequence and investigates the effect of persistent excitation and forgetting on $P_k$.
For all $j\ge0$ and $k\ge j,$ define    %
    \begin{align}
        F_{j,k}
                \isdef
                    \sum_{i=j}^{k}
                        \phi_i^\rmT  \phi_i.
    \end{align}

\begin{defin}
    \label{def:persistent_Exc}
    The sequence $(\phi_k)_{k=0}^\infty \subset \BBR^{p \times n}$ is \mbox{\it persistently exciting} if there exist $N \ge  n/p $ and $\alpha, \beta \in(0,\infty)$ such that, for all $j\ge0,$
    \begin{align}
        \alpha I_n
            \le
                F_{j,j+N}
            \le 
                    \beta I_n.
        \label{eq:persistent_def}
    \end{align}
\end{defin}

Suppose that $(\phi_k)_{k=0}^\infty$ is persistently exciting and \eqref{eq:persistent_def} is satisfied for given values of $N,\alpha,\beta.$
Then, with suitably modified values of $\alpha$ and $\beta,$ \eqref{eq:persistent_def} is satisfied for all larger values of $N$. 
For example, if $N$ is replaced by $2N,$ then \eqref{eq:persistent_def} is satisfied with $\alpha$ replaced by $2\alpha$ and $\beta$ replaced by $2\beta.$
The following result expresses \eqref{eq:Pk_recursive} in terms of $F_{0,k}$ in the case where $\lambda=1.$

\begin{lem} \label{lem:Pk_F0k}
    Let $\lambda = 1$ and, for all $k\ge1,$ define $P_k$ as in Theorem \ref{theorem_RLS}.
    Then,
    \begin{align}
        P_k^{-1} = F_{0,k} + P_0^{-1}.\label{PkFkP0}
    \end{align}
\end{lem}

The following result shows that, if $\SeqPhi$ is persistently exciting and $\lambda = 1$, then $P_k$ converges to zero.  
%

\begin{prop}
\label{prop:Pkinv_bounds_wo_forgetting}
Assume that $(\phi_k)_{k=0}^\infty \in \BBR^{p \times n}$ is persistently exciting,
let $N,\alpha,\beta$ be given by Definition \ref{def:persistent_Exc}, let $R\in\BBR^{n\times n}$ be positive definite, define $P_{0} \isdef R^{-1}$, let $\lambda =1$, and, for all $k\ge0,$ let $P_k$ be given by \eqref{eq:PUpdate}.
Then, for all $k\ge N+1,$ 
    \begin{align}
        \left \lfloor \tfrac{k}{N+1} \right \rfloor \alpha I_n + P_0^{-1} 
            \le
        P_k^{-1}
            \le 
        \left \lceil \tfrac{k}{N+1} \right \rceil \beta I_n + P_0^{-1} .
        \label{eq:Pkinv_bounds_wo_forgetting}
    \end{align}
Furthermore, 
    \begin{align}
        \lim_{k \to \infty} P_k = 0.
        \label{eq:Pk_limit_forgetting}
    \end{align}
\end{prop}


\begin{proof}
    First, note that, for all $k \ge 0$,
    \begin{align}
        F_{0,k}
            &=
                \sum_{i=1}^{\left \lfloor \tfrac{k}{N+1} \right \rfloor } F_{(i-1)(N+1), i(N+1)-1} 
                + 
                F_{\left \lfloor \tfrac{k}{N+1} \right \rfloor (N+1) ,k} \nn \\
            &\le
                \sum_{i=1}^{\left \lceil \tfrac{k}{N+1} \right \rceil } F_{(i-1)(N+1), i(N+1)-1} \nn ,
    \end{align}
    and thus  \eqref{eq:persistent_def} implies that 
    %
    %
    \begin{align}
        \left \lfloor \tfrac{k}{N+1} \right \rfloor \alpha I_n
            &\le
        \sum_{i=1}^{\left \lfloor \tfrac{k}{N+1} \right \rfloor }  F_{(i-1)(N+1), i(N+1)-1} 
        \nn \\
            &\le
        \sum_{i=1}^{\left \lceil \tfrac{k}{N+1} \right \rceil }  F_{(i-1)(N+1), i(N+1)-1} \nn \\
            &\le
        \left \lceil \tfrac{k}{N+1} \right \rceil \beta I_n.
        \label{eq:F_jN_ineq3}
    \end{align}
    It follows from Lemma \ref{lem:Pk_F0k} 
    and \eqref{eq:F_jN_ineq3} that, for all $k\ge N+1,$
    \begin{align}
        \left \lfloor \tfrac{k}{N+1} \right \rfloor \alpha I_n + P_0^{-1}
            &\le
                F_{0,\left \lfloor \tfrac{k}{N+1} \right \rfloor(N+1)-1} + P_0^{-1}
            \nn \\
            &\le
                F_{0,k} + P_0^{-1}
            \nn \\
            &=
                P_k^{-1}  \nn \\
            &\le
                F_{0,\left \lceil \tfrac{k}{N+1} \right \rceil (N+1)-1} + 
                P_0^{-1}
            \nn \\
            &\le
                \left \lceil \tfrac{k}{N+1}\right \rceil \beta  I_n + P_0^{-1}. \nn
    \end{align}
    Finally, it follows from \eqref{eq:Pkinv_bounds_wo_forgetting} that $\lim_{k \to \infty} P_k = 0$.
\end{proof}

The following example shows that $\lim_{k \to \infty} P_k = 0$ does not imply that $\SeqPhi$ is persistently exciting.

\begin{exam} \label{exam:CounterExample}
    \textit{$P_k$ converges to zero without persistent excitation.}
    For all $k\ge0$, let $\phi_k = \tfrac{1}{\sqrt{k+1}}$.
    Let $\lambda = 1$.
    For all $N\ge 1$, note that $F_{j,j+N} \le \tfrac{N+1}{j+1}$, and thus there does not exist $\alpha$ satisfying \eqref{eq:persistent_def}. 
    Hence, $\SeqPhi$ is not persistently exciting. 
    %
    However, it follows from \eqref{eq:Pk_recursive} that, for all $k \ge 0$,
    \begin{align}
        P_k^{-1} 
            &=
                \sum_{i = 0}^{k} \dfrac{1}{i+1} + P_0^{-1}.
    \end{align}
    Thus, $\lim_{k \to \infty} P_k = 0$. 
    %
    \EndExample
\end{exam}

The following result given in \cite{johnstone1982exponential} shows that, if $\SeqPhi$ is persistently exciting and $\lambda \in (0,1)$, then $P_k$ is bounded. 
%

\begin{prop}
\label{prop:Pkinv_bounds}
Assume that $(\phi_k)_{k=0}^\infty \in \BBR^{p \times n}$ is persistently exciting,
let $N,\alpha,\beta$ be given by Definition \ref{def:persistent_Exc}, let $R\in\BBR^{n\times n}$ be positive definite, define $P_{0} \isdef R^{-1}$, let $\lambda \in (0,1)$, and, for all $k\ge0,$ let $P_k$ be given by \eqref{eq:PUpdate}.
Then, for all $k\ge N+1,$ 
    \begin{align}
        \frac{\lambda^N(1-\lambda) \alpha} 
             {1 - \lambda^{N+1} }
             I_n
            \le
                P_k^{-1}
            \le 
            \dfrac{\beta}{1-\lambda^{N+1}}  I_n
            +
            P_{N}^{-1}.
        \label{eq:Pkinv_bounds}
    \end{align}
\end{prop}

\begin{proof}
It follows from \eqref{eq:Pk_recursive} that, for all $i\ge0$,
$\lambda P_{i}^{-1} \le P_{i+1}^{-1}  $  
and 
$\phi_i^\rmT \phi_i \le P_{i+1}^{-1},$
and thus, for all $i,j\ge0,$  
$ \lambda^{j}P_{i}^{-1} \le P_{i+j}^{-1}.$  
Hence, for all $k \ge N+1$,
%
%
\begin{align}
    \alpha I_n
            &\le
                \sum_{i=k-N-1}^{k-1}
                    \phi_i^\rmT  \phi_i\nn\\
        &\le 
            \sum_{i=k-N}^{k} P_{i}^{-1}\nn\\
            &\le
            ( \lambda^{-N} + \cdots + 1 ) P_{k}^{-1}\nn\\
        &=
            \dfrac{1-\lambda^{N+1}}{\lambda^N (1-\lambda)} P_{k}^{-1}, \nn 
\end{align}
which proves the first inequality in \eqref{eq:Pkinv_bounds}.
To prove the second inequality in \eqref{eq:Pkinv_bounds}, note that, for all $k\ge N+1$,
\begin{align} 
    P_k^{-1}
        &\le 
            \dfrac{1-\lambda}{1-\lambda^{N+1}}
            \sum_{i=k-1}^{k+N-1} P_{i+1}^{-1}  \nn\\
        &\le 
            \dfrac{1-\lambda}{1-\lambda^{N+1}}
            \left(
                \lambda \sum_{i=k-1}^{k+N-1} P_{i}^{-1} + \beta I_n  
            \right)
            \nn\\
        &\le 
            \dfrac{1-\lambda}{1-\lambda^{N+1}}
            \left(
                \lambda^{k} \sum_{i=0}^{N} P_{i}^{-1} + \dfrac{1-\lambda^{k}}{1-\lambda} \beta I_n 
            \right)
            \nn\\
            %
        &\le
            \lambda^{k-N} P_{N}^{-1} + 
            \dfrac{(1-\lambda^{k})\beta }{1-\lambda^{N+1}}  I_n.
            \nn \\
        &\le
            P_{N}^{-1} + 
            \dfrac{\beta }{1-\lambda^{N+1}}  I_n.
            \nn
    \end{align}
\end{proof}


%

The next result, which is an immediate consequence of \eqref{eq:Pk_recursive}, is a converse of Proposition 
\ref{prop:Pkinv_bounds}.

\begin{prop}
\label{prop:Pkinv_boundsconv}
Define $\phi_k,$ $y_k,$ $R,$ and $P_0$ as in Theorem \ref{theorem_RLS}, let $\lambda \in (0,1)$, and let $P_k$ be given by \eqref{eq:PUpdate}.
Furthermore, assume there exist $\overline{\alpha},\overline{\beta}\in(0,\infty)$ such that, for all $k\ge0,$
$\overline{\alpha} I_n \le P_k^{-1} \le \overline{\beta} I_n.$
Let $N \ge \tfrac{\lambda \overline{\beta} - \overline{\alpha} }{(1-\lambda) \overline \alpha}$.
Then, for all $j \ge 0$,
\begin{align}
    [(1 +  (1 - \lambda) N )\overline{\alpha} - \lambda \overline{\beta} ] I_n
        \le
            \sum_{i=j}^{j+N} \phi_i^\rmT \phi_i
        \le
            \dfrac{1 - \lambda^{N+1}}{ \lambda^{N}(1-\lambda)} \overline{\beta}I_n.
    \label{eq:Pkinv_bounds_converse}
\end{align}
Consequently, $(\phi_k)_{k=0}^\infty$ is persistently exciting.
\end{prop}

\begin{proof}
    Note that, for all $j \ge 0$,
    \begin{align}
        [(1 +  (1 - \lambda) N )\overline{\alpha} - \lambda \overline{\beta} ] I_n
            &=
                \overline \alpha I_n + 
                (1 - \lambda) N \overline{\alpha} I_n -
                \overline{ \beta} I_n
            \nn \\
            &\le
                 P_{j+N+1}^{-1} +
                (1 - \lambda) \sum_{i=j+1}^{j+N} P_{i}^{-1} -
                \lambda P_j^{-1}
            \nn \\
            &=
                \sum_{i=j}^{j+N} ( P_{i+1}^{-1} - \lambda P_{i}^{-1})
            \nn \\
            &=
                \sum_{i=j}^{j+N} \phi_i^\rmT \phi_i,
            \nn
    \end{align}
    which proves the first inequality in \eqref{eq:Pkinv_bounds_converse}.
    To prove the second inequality in \eqref{eq:Pkinv_bounds_converse}, note that  \eqref{eq:Pk_recursive} implies that, for all $i\ge0$, $\lambda P_{i}^{-1} \le P_{i+1}^{-1}  $ and  $\phi_i^\rmT \phi_i \le P_{i+1}^{-1},$
    and thus, for all $i,j\ge0,$  
    $ \lambda^{j}P_{i}^{-1} \le P_{i+j}^{-1}.$  
    Hence, for all $j \ge 0$,
    \begin{align}
        \sum_{i=j}^{j+N} \phi_i^\rmT \phi_i
                &\le
                    \sum_{i=j}^{j+N}
                        P_{i+1}^{-1}
                \nn \\
                &\le
                    (\lambda^{-N} +  \cdots + 1) P_{j+N+1}^{-1}
                \nn \\
                &\le
                    \dfrac{1 - \lambda^{N+1}}{ \lambda^{N}(1-\lambda)} \overline{\beta}I_n.
                \nn
    \end{align}
    Finally, it follows from Definition \ref{def:persistent_Exc} with
    $N \ge \tfrac{\lambda \overline{\beta} - \overline{\alpha} }{(1-\lambda) \overline \alpha},$
    $\alpha = (1 +  (1 - \lambda) N )\overline{\alpha} - \lambda \overline{\beta},$ and
    $\beta = \tfrac{1 - \lambda^{N+1}}{ \lambda^{N}(1-\lambda)} \overline{\beta},$ 
    that $\SeqPhi$ is persistently exciting. 
\end{proof}

The proof of Proposition \ref{prop:Pkinv_boundsconv} shows that the condition $N \ge \tfrac{\lambda \overline{\beta} - \overline{\alpha} }{(1-\lambda) \overline \alpha}$ is needed to satisfy the lower bound in Definition \ref{def:persistent_Exc}. However, the upper bound in Definition \ref{def:persistent_Exc} is satisfied for all $N \ge 1$.

\begin{exam}
\label{exam:PE_Pk_bounds_1}
    \textit{Persistent excitation and bounds on $P_k^{-1}$}.
    Let $\phi_k=[u_k\ \ u_{k-1}],$
    where  $u_k$ is the periodic signal
         \begin{align}
        u_k
            =
                \sin \frac{2 \pi k}{17}+
                \sin \frac{2 \pi k}{23}+
                \sin \frac{2 \pi k}{53}.
        \label{eq:u_harm}
    \end{align}
    Figure \ref{fig:CSM_forgetting_Prop1_limits} shows the singular values of $F_{j,j+N}$ for $N=2$ and $N=10$, as well as the singular values of $P_k^{-1}$ with the corresponding upper and lower bounds given by \eqref{eq:Pkinv_bounds} for $N=2$ and $N=10$. 
    \EndExample
\end{exam}

\begin{figure}[h!]
    \centering
    \includegraphics[width=1 \textwidth]{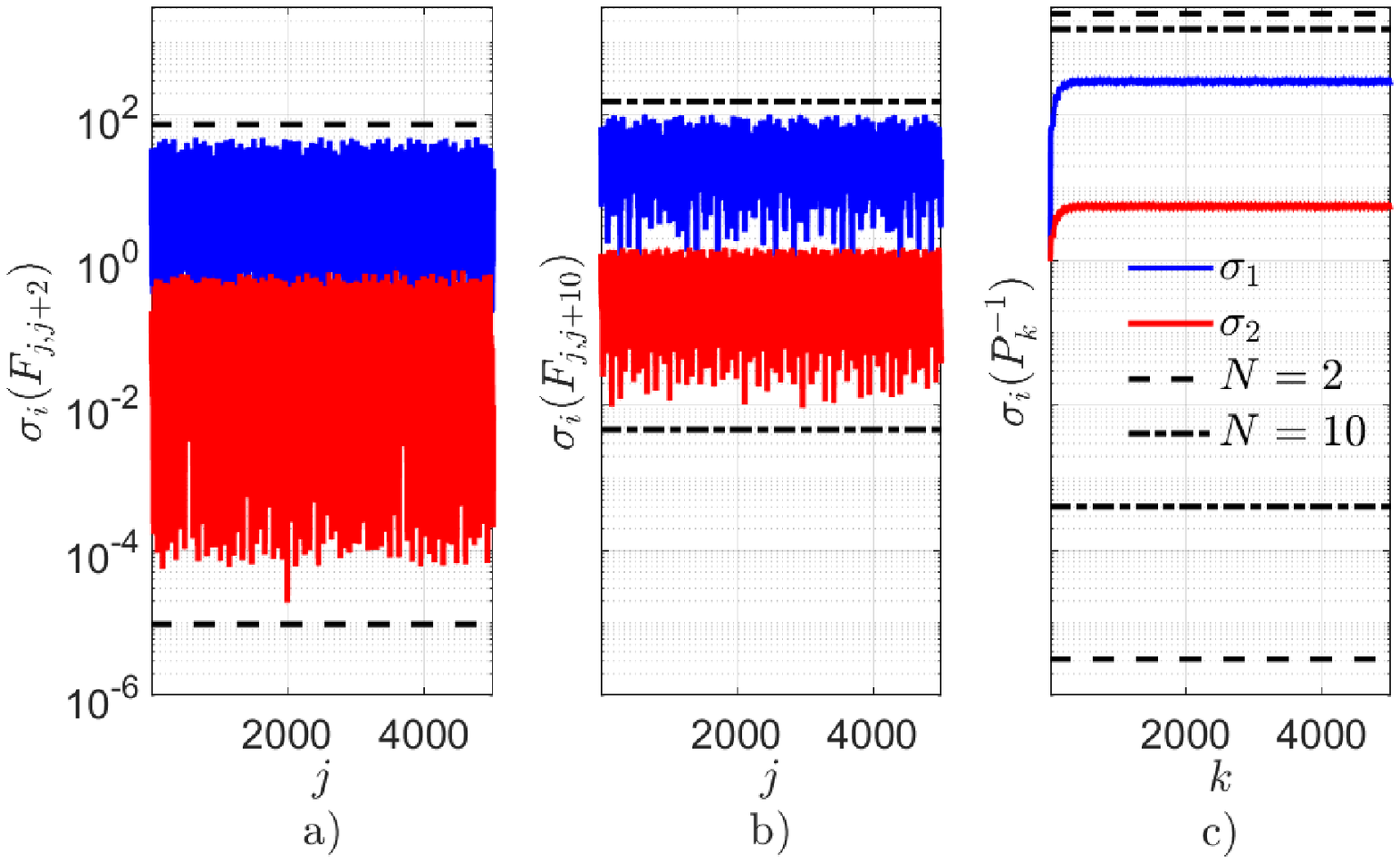}
    \caption
        {
            Example \ref{exam:PE_Pk_bounds_1}. 
            Persistent excitation and bounds on $P_k^{-1}$.   
            a) and b) show the singular values of $F_{j,j+N}$ for $N=2$ and $N=10$, where $\alpha$ and $\beta$ are chosen to satisfy \eqref{eq:persistent_def}.
            Since $u_k$ is periodic, it follows that, for all $j\ge0,$ the lower and upper bounds  \eqref{eq:persistent_def} for $F_{j,j+N}$ are satisfied.   Hence, $\SeqPhi$  is persistently exciting.
            c) shows the singular values of $P_k^{-1}$, with corresponding bounds given by \eqref{eq:Pkinv_bounds} for $\lambda = 0.99$.
            Note that $\alpha$ and $\beta$ are larger for $N=10$ than for $N=2$, as expected.
        }
    \label{fig:CSM_forgetting_Prop1_limits}
\end{figure}

\begin{exam}
\label{exam:PE_Pk_bounds_2}
    \textit{Lack of persistent excitation and bounds on $P_k^{-1}$}.
    Let $\phi_k=[u_k\ \ u_{k-1}],$
    where  $u_k$ is given by \eqref{eq:u_harm} for all $k<2500$ and $u_k = 1$ for all $k \ge 2500$.
    Figure \ref{fig:CSM_forgetting_Prop1_limits_NotPersistent} shows the singular values of $F_{j,j+2}$ and 
    the singular values of $P_k^{-1}$ for $\lambda = 1$ and $\lambda = 0.9$, respectively.
    %
    %
    Note that, for $\lambda = 1$, one of the singular values of $P_k^{-1}$ diverges, whereas, for $\lambda \in (0,1)$, one of singular values of $P_k^{-1}$ converges to zero. 
    %
    %
    \EndExample
\end{exam}

\begin{figure}[h!]
    \centering
    \includegraphics[width=1 \textwidth]{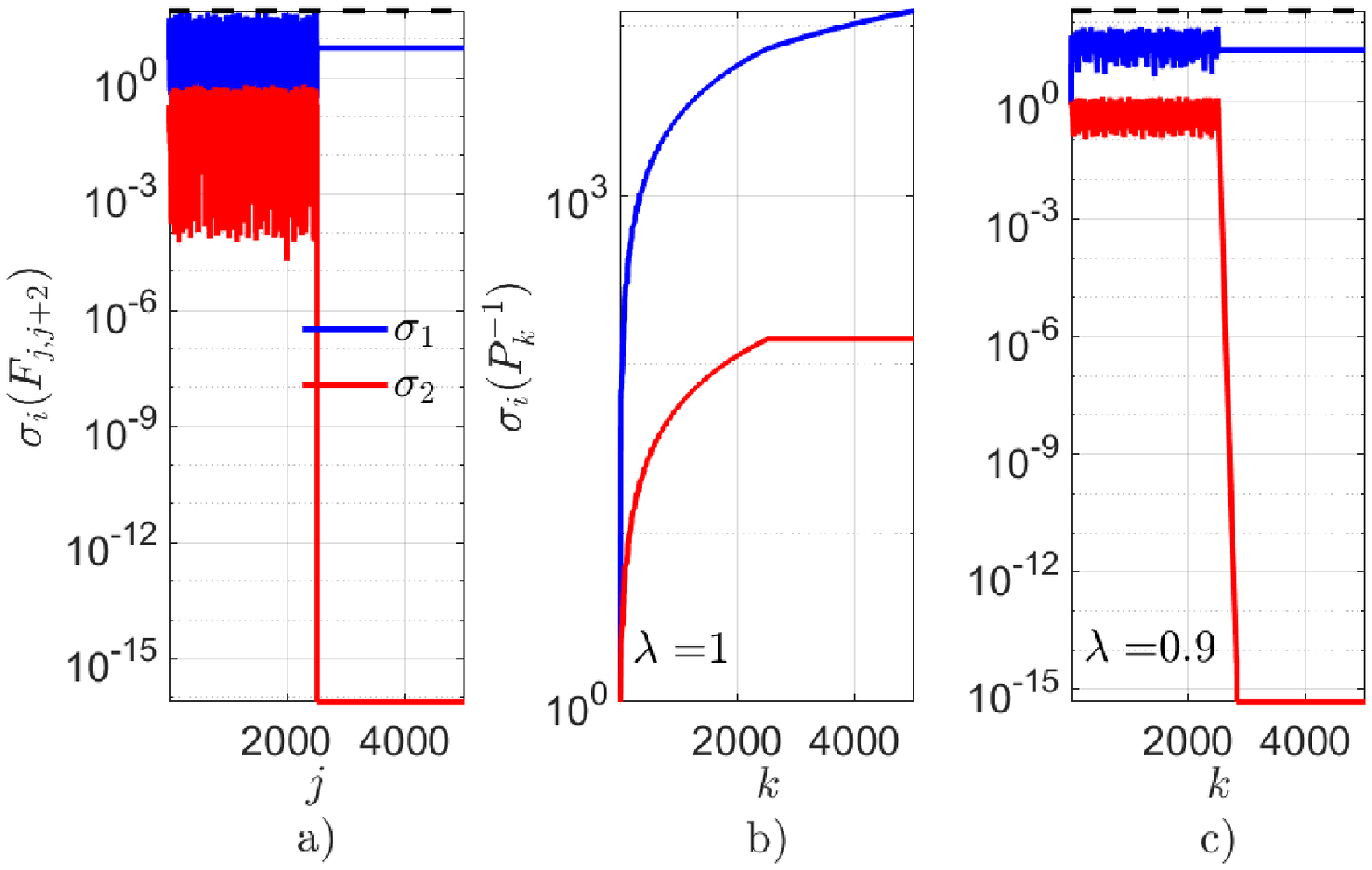}
    \caption
        {
            Example \ref{exam:PE_Pk_bounds_2}. 
            Lack of persistent excitation and bounds on $P_k^{-1}$. 
            a) shows the singular values of $F_{j,j+2}$.
            Note that the smaller singular value of $F_{j,j+2}$  reaches zero in machine precision, and thus that $\alpha >0$ satisfying \eqref{eq:persistent_def} does not exist.
            Hence, $\phi_k$ is not persistently exciting. 
            The upper bound $\beta$ shown by the dashed line is chosen to satisfy \eqref{eq:persistent_def}.
            b) and c) show the  singular values of $P_k^{-1}$ for $\lambda = 1$ and $\lambda = 0.9$, respectively.
            Note that, if $\lambda = 1$, then one of the singular values of $P_k^{-1}$ diverges, whereas, if $\lambda \in (0,1)$, then one of singular values of $P_k^{-1}$ converges to zero. 
        }
    \label{fig:CSM_forgetting_Prop1_limits_NotPersistent}
\end{figure}

The following result shows that the \textit{predicted error} $z_k \isdef \phi_k \theta_k - y_k$ converges to zero whether or not $\SeqPhi$ is persistent.

\begin{prop}
    \label{prop:z_converges}
    For all $k\ge0$, let $\phi_k \in \BBR^{p \times n}$ and $y_k \in \BBR^p$, let $R\in\BBR^{n\times n}$ be positive definite, and let $P_0 = R^{-1}$, $\theta_0 \in \BBR^n$, and  $\lambda \in (0,1]$.
    Furthermore, for all $k\ge0,$ let $P_k$ and $\theta_k$ be given by \eqref{eq:PUpdate} and \eqref{eq:thetaUpdate}, respectively, and 
    define the \textit{predicted error} $z_k \isdef \phi_k \theta_k - y_k.$
    Then, 
    \begin{align}
        \lim_{k \to \infty} z_k = 0.
        \label{eq:z_lim}
    \end{align}
\end{prop}

\begin{proof}  
    For all $k\ge0$, note that $z_k = \phi_k \tilde \theta_k,$ and define $V_k \isdef \tilde \theta_k^\rmT P_k^{-1} \tilde \theta_k$.
    Note that, for all $k\ge 0$ and $\tilde \theta_k \in \BBR^n$, $V_k \ge 0$.
    Furthermore, for all $k\ge0,$
    \begin{align}
        V_{k+1} - V_k
            &=
                \tilde \theta_{k+1}^\rmT P_{k+1}^{-1} \tilde \theta_{k+1} -
                \tilde \theta_k^\rmT P_k^{-1} \tilde \theta_k 
            \nn \\
            &=
                \lambda^2   \tilde \theta_k^\rmT P_k^{-1} P_{k+1} P_k^{-1} \tilde \theta_{k} -
                \tilde \theta_k^\rmT P_k^{-1} \tilde \theta_k 
            \nn \\
            &=
                (
                \lambda \tilde \theta_{k+1}^\rmT -
                \tilde \theta_k^\rmT
                )
                P_k^{-1} \tilde \theta_k 
            \nn \\
            &=
                -[(1-\lambda ) \tilde \theta_k^\rmT + 
                \lambda \tilde \theta_k^\rmT \phi_k^\rmT \phi_k P_{k+1} 
                ]
                P_k^{-1} \tilde \theta_k 
            \nn \\
            &=
                -[(1-\lambda ) \tilde \theta_k^\rmT
                P_k^{-1} \tilde \theta_k +
                \lambda \tilde \theta_k^\rmT \phi_k^\rmT \phi_k P_{k+1} 
                P_k^{-1} \tilde \theta_k 
                ]
            \nn \\
            &=
                -[
                (1-\lambda ) \tilde \theta_k^\rmT
                P_k^{-1} \tilde \theta_k +
                \tilde \theta_k^\rmT \phi_k^\rmT 
                [
                    I_p- \phi_k P_{k} \phi_k^\rmT
                    ( \lambda I_p + \phi_k  P_{k} \phi_k^\rmT )^{-1}
                ]
                \phi_k
                \tilde \theta_k ]
            \nn \\
            &=
                -
                [
                (1-\lambda ) V_k +
                z_k^\rmT 
                [
                    I_p- \phi_k P_{k} \phi_k^\rmT
                    ( \lambda I_p + \phi_k  P_{k} \phi_k^\rmT )^{-1}
                ]
                z_k 
                ]
            \nn \\
            &\le 
                0.
            \nn
    \end{align}
    Note that, since $(V_k)_{k=1}^\infty$ is a nonnegative, nonincreasing sequence, it converges to a nonnegative number.
    Hence, $\lim_{k\to\infty} (V_{k+1} - V_k)=0,$ which implies that $\lim_{k\to\infty} [ (1-\lambda ) V_k + z_k^\rmT  R_k z_k  ] = 0,$
    where $R_k \isdef I_p- \phi_k P_{k} \phi_k^\rmT ( \lambda I_p + \phi_k  P_{k} \phi_k^\rmT )^{-1} $.
    Lemma \ref{prop:IminusAlambda} from ``Three Useful Lemmas'' implies that $R_k$ is positive definite.
    Since $V_k \ge 0$, it follows that $\lim_{k \to \infty} z_k = 0.$
    %
\end{proof}

The following example shows that $\theta_k$ may converge despite the fact that $\SeqPhi$ is not persistent. 
%
%
\begin{exam}
    \label{exam:LoP_zto0_theta_conv}
    \textit{Convergence of $z_k$ and $\theta_k$.}
    Consider the first-order system
    \begin{align}
        y_k
            =
                \dfrac
                    {0.8}
                    {\shiftq - 0.4}
                u_k,
        \label{eq:1G}
    \end{align}
    where $\shiftq$ is the forward-shift operator.
    Define
    $
        \phi_{k}
            \isdef
                [y_{k-1} \ u_{k-1} ],
    $
    so that $y_k = \phi_k \theta$, where $\theta$ consists of the coefficients in \eqref{eq:1G}.
    %
    To apply RLS, let $P_0 = I_{2}$, $\theta_0 = 0$, and $\lambda = 0.999$.
    Figure \ref{fig:CSM_forgetting_LoP_zto0_theta_conv} shows the shows the singular values of $F_{j,j+10}$,
    the predicted error $z_k$, and
    the parameter estimate $\theta_k$ for two choices of the input $u_k$.
    In the first case, for all $k\ge 0$, $u_k=1$, whereas 
    in the second case, for all $k\ge 0$, $u_k=1$.
    For both choices of $u_k$, the predicted error $z_k$ converges to zero, which confirms Proposition \ref{prop:z_converges}, and $\theta_k$ converges. 
    Note that, in these two cases, $\theta_k$ converges to different parameter values, neither of which is the true value.
    \EndExample
\end{exam}

\begin{figure}[h!]
    \centering
    \includegraphics[width=1 \textwidth]{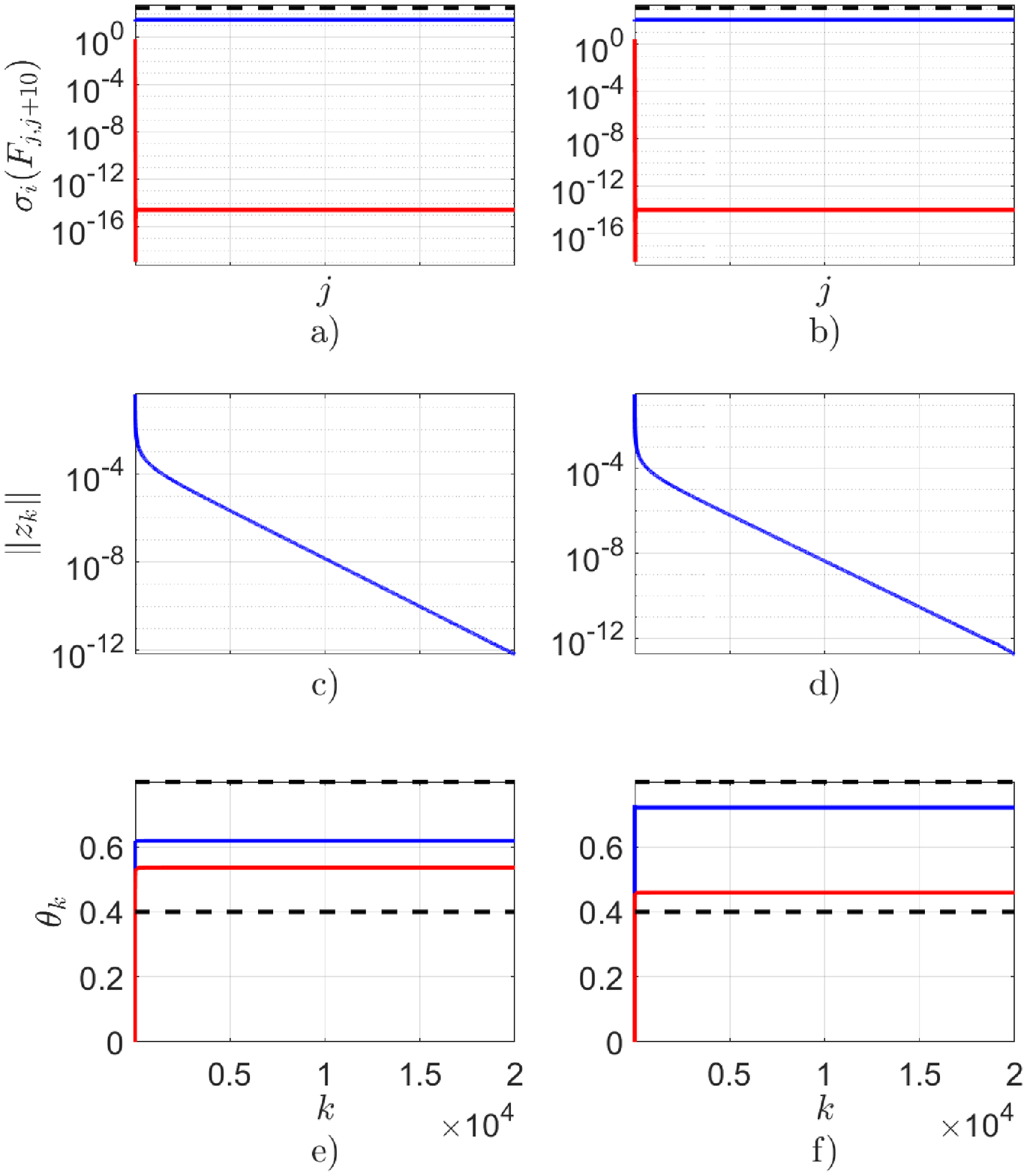}
    \caption
        {
            Example \ref{exam:LoP_zto0_theta_conv}. 
            Convergence of $z_k$ and $\theta_k$. 
            a) and b) show the singular values of $F_{j,j+10}$ for two choices of $u_k$.
            Note that the singular value of $F_{j,j+10}$ that is close to machine precision $(\approx 10^{-15})$ is essentially zero.  
            Definition \ref{def:persistent_Exc} thus implies that $\SeqPhi$ is not persistently exciting. 
            c) and d) show the predicted error $z_k$ for both cases.
            Note that $z_k$ converges to zero in both cases. 
            Finally, e) and f) show the parameter estimate $\theta_k$ for both cases.
            Note that, for both choices of input $u_k$, $\theta_k$ converge, but to different parameter values. 
        }
    \label{fig:CSM_forgetting_LoP_zto0_theta_conv}
\end{figure}

Table \ref{tab:P_k_behavior} summarizes the results in this section.

\begin{table}[ht]
    \caption{Behavior of $P_k$ with and without persistent excitation.}
    \centering
    \begin{tabular}{|l|p{5cm}|p{5cm}|}
        \hline
        Excitation $\backslash$ $\lambda $  &  $\lambda = 1$ & $\lambda \in (0,1)$ \\
        \hline
            Persistent   &  
            \tabitem $P_k$ converges to zero &
            \tabitem $P_k$ is bounded  \\
            
                        &
            \tabitem Proposition \ref{prop:Pkinv_bounds_wo_forgetting} &
            \tabitem Propositions \ref{prop:Pkinv_bounds}, \ref{prop:Pkinv_boundsconv}  \\
                        &
            \tabitem Example \ref{exam:PE_Pk_bounds_1} &
            \tabitem Example \ref{exam:PE_Pk_bounds_1}  \\
        \hline
            Not Persistent & 
            \tabitem All singular values of $P_k$ are bounded & 
            \tabitem Some singular values of $P_k$ diverge \\
                        &
            \tabitem Some of these  converge to zero &
            \tabitem The remaining singular values are bounded  \\    
                        &
            \tabitem Example \ref{exam:PE_Pk_bounds_2} &
            \tabitem Example \ref{exam:PE_Pk_bounds_2}  \\
        \hline
    \end{tabular}
    \label{tab:P_k_behavior}
\end{table}

\section{Persistent Excitation and the Condition Number}

For nonsingular $A\in\BBR^{n\times n},$  the condition number of $A$ is defined by
\begin{align}
    \kappa (A) 
        \isdef 
            \dfrac{\sigma_{\rm max}(A)}{ \sigma_{\rm min}(A)},
\end{align}
%
For $B \in \BBR^{n \times m},$ let $\|B\|$ denotes the maximum singular value of $B$.
If $A$ is positive definite, then 
\begin{align}
    \| A^{-1}\|^{-1} I_n 
        =
    \sigma_{\rm min}(A) I_n 
        \le 
            A
        \le
    \sigma_{\rm max}(A) I_n
        =
    \| A\| I_n .
\end{align}
Therefore, if $\alpha,\beta\in(0,\infty)$ satisfy $\alpha \le \sigma_{\rm min}(A)$ and $\sigma_{\rm max}(A) \le \beta$,
then $\kappa(A) \le {\beta}/{\alpha}$.
Thus, if $\lambda = 1$ and $(\phi_k )_{k=0}^\infty$ is persistently exciting with $N,\alpha, \beta$ given by Definition \ref{def:persistent_Exc}, then   \eqref{eq:Pkinv_bounds_wo_forgetting} implies that 
\begin{align}
    \kappa(P_k)
        \le 
             \dfrac {\beta}{\alpha}.
\end{align}
Similarly, if $\lambda \in (0,1)$ and $(\phi_k )_{k=0}^\infty$ is persistently exciting with $N,\alpha, \beta$  given by Definition \ref{def:persistent_Exc}, then \eqref{eq:Pkinv_bounds} implies that 
\begin{align}
    \kappa(P_k)
        \le 
             \dfrac
                {\beta + (1-\lambda^{N+1}) \| P_N^{-1} \|}
                {\lambda^N (1-\lambda) \alpha}.
\end{align}
However, as shown by Example \ref{exam:PE_Pk_bounds_2}, in the case where $\SeqPhi$ is not persistently exciting, there might not exist $\alpha>0$ satisfying \eqref{eq:persistent_def}, and thus $\kappa(P_k)$ cannot be bounded.
Hence $\kappa(P_k)$ can be used to determine whether or not $\SeqPhi$ is persistently exciting, where a bounded condition number implies that $\SeqPhi$ is persistently exciting, and a diverging condition number implies that $\phi_k$ is not persistently exciting, as illustrated by the following example.
\cite{Benesty2004} provides a recursive algorithm for computing $\kappa(P_k)$.

\begin{exam}
    \label{exam:persistency_conditionNumber}
    \textit{Using the condition number of $P_k$ to determine whether $\SeqPhi$ is persistently exciting.}
    Consider the 5th-order system
    \begin{align}
        y_k
            =
                \dfrac
                    {0.68 \shiftq^4 - 0.16 \shiftq^3 - 0.12 \shiftq^2 - 0.18 \shiftq + 0.09}
                    {\shiftq^5 - \shiftq^4 + 0.41 \shiftq^3 - 0.17 \shiftq^2 - 0.03  \shiftq + 0.01}
                u_k,
        \label{eq:5thOrderG}
    \end{align}
    where $u_k$ is given by \eqref{eq:u_harm}.
    To apply RLS, let $\theta$ consist of the coefficients in \eqref{eq:5thOrderG} and let 
    \begin{align}
        \phi_{k}
            &=
                [u_{k-1} \ \cdots \ u_{k-5} \ 
                y_{k-1} \ \cdots \ y_{k-5} ],\label{eq:Reg_IIR}
    \end{align} 
    so that $y_k = \phi_k\theta.$
    Letting $P_0 = I_{10}$, 
    Figure \ref{fig:CSM_forgetting_Persistency_harm} shows the singular values of $F_{j,j+20}$ and
    the singular values and  condition number of $P_k$ for $\lambda = 1$ and $\lambda = 0.99$.
    In particular, the smallest singular value of $F_{j,j+20}$ is essentially zero, which indicates that $\SeqPhi$ is not persistently exciting.
    Consequently, in the case where $\lambda = 0.99$, $P_k$ becomes ill-conditioned. %
    \EndExample
\end{exam}
\begin{figure}[h!]
    \centering
    \includegraphics[width=1\textwidth]
    {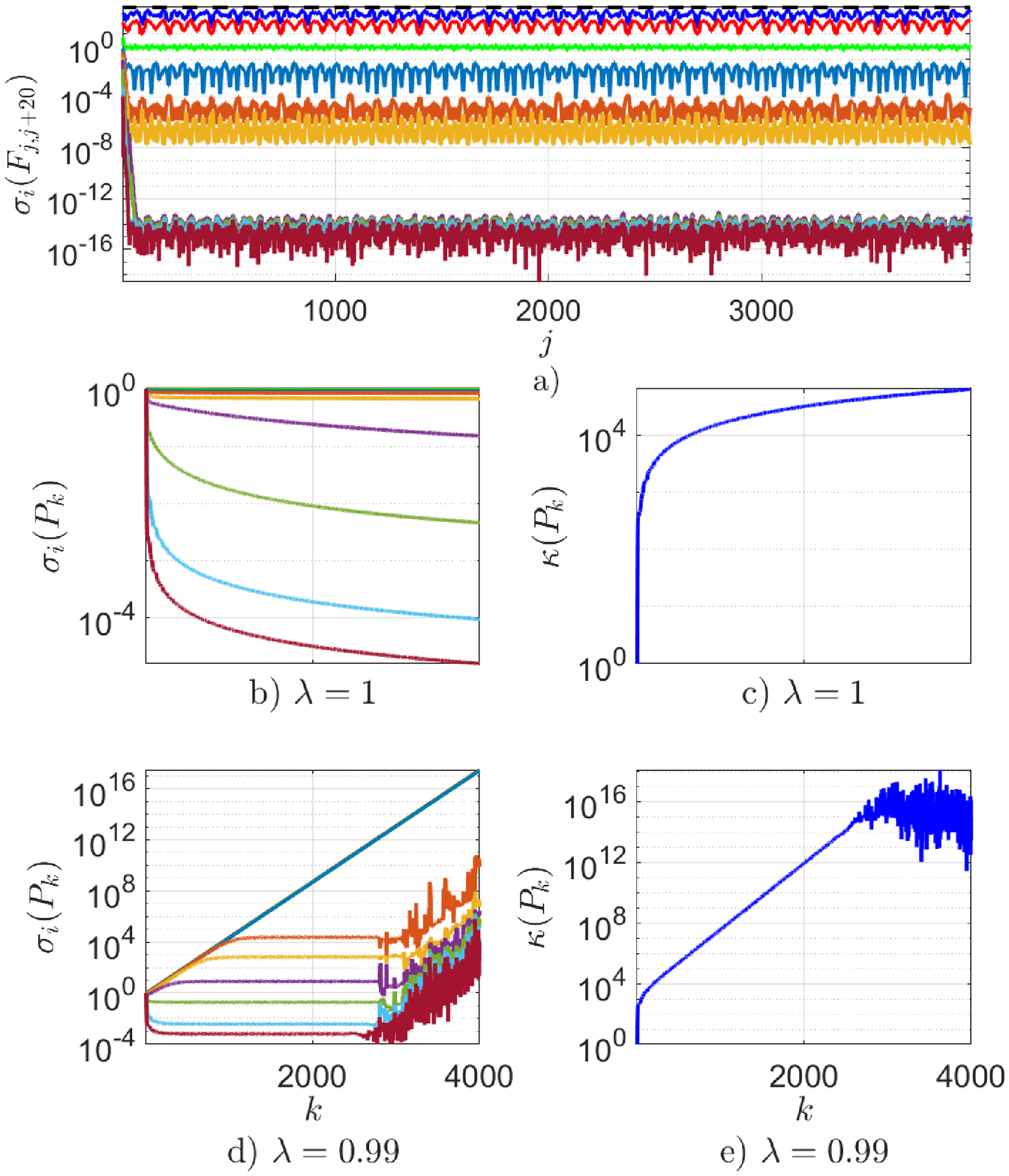}
    \caption
        {
        Example \ref{exam:persistency_conditionNumber}.
        Using the condition number of $P_k$ to evaluate persistency.
        a) shows the singular values of $F_{j,j+20}$, where the singular values of $F_{j,j+20}$ close to machine precision $(\approx 10^{-15})$ are essentially zero, thus implying that $\SeqPhi$ is not persistently exciting. 
        b) and c) shows the singular values and the condition number of $P_k$ for $\lambda = 1$.
        Note that the six singular values of $P_k$ decrease due to the presence of three harmonics in $u_k$. 
        d) and e) shows the singular values and the condition number of $P_k$ for $\lambda = 0.99$.
        Note that the six singular values of $P_k$ remain bounded due to the presence of three harmonics in $u_k$. 
        However, $P_k$ becomes ill-conditioned due to the lack of persistent excitation.
        }
    \label{fig:CSM_forgetting_Persistency_harm}
\end{figure}

In Example \ref{exam:persistency_conditionNumber}, the regressor $\SeqPhi$ is not persistently exciting.
Consequently, in the case where $\lambda = 1$, it follows from \eqref{PkFkP0} that $P_k$ is bounded by $P_0 $, and thus all of the singular values of  $P_k$ are bounded; this property is illustrated by Figure \ref{fig:CSM_forgetting_Persistency_harm}.
However, Figure \ref{fig:CSM_forgetting_Persistency_harm} also shows that not all of the singular values of $P_k$ converge to zero.
%
%
%
On the other hand, in the case where $\lambda = 0.99$, Figure \ref{fig:CSM_forgetting_Persistency_harm} shows that some of the singular values of $P_k$ are bounded, whereas the remaining singular values diverge. 
This example thus shows that singular values can diverge due to the lack of persistent excitation with $\lambda\in(0,1).$
%
%


\section{Lyapunov Analysis of the Parameter Error}

Let $k\ge0,$ and consider the system
\begin{align}
    x_{k+1} = f(k,x_k),
    \quad
    \label{eq:NonLinSys}
\end{align}
where $x_k \in \BBR^n,$  $f\colon \{0,1,2, \ldots \} \times \BBR^n\to\BBR^n$ is continuous, and, for all $k \ge 0$, $f(k,0) = 0$.
Let $ \SD \subset \BBR^n$ be an open set such that $0 \in \SD$.
%

\begin{defin}
    \label{def:LS1}
    The zero solution of \eqref{eq:NonLinSys} is \textit{Lyapunov stable} if, 
    for all $\varepsilon > 0$ and $k_0  \ge 0$, 
    there exists $\delta(\varepsilon, k_0) > 0$ such that, for all $x_{k_0} \in \BBR^n$ satisfying 
    $\| x_{k_0} \| < \delta(\varepsilon, k_0)$, it follows that, for all $k \ge k_0$,  $\| x_k \| < \varepsilon$.
\end{defin}

\begin{defin}
    \label{def:ULS1}
    The zero solution of \eqref{eq:NonLinSys} is \textit{uniformly Lyapunov stable} if, 
    for all $\varepsilon > 0$, 
    there exists $\delta(\varepsilon) > 0$ such that, for all $k_0  \ge 0$ and all $x_{k_0} \in \BBR^n$ satisfying 
    $\| x_{k_0} \| < \delta(\varepsilon)$, it follows that, for all $k \ge k_0$,  $\| x_k \| < \varepsilon$.
\end{defin}

\begin{defin}
    \label{def:GAS1}
    The zero solution of \eqref{eq:NonLinSys} is \textit{globally asymptotically  stable} if it is Lyapunov stable and, 
    for all $k_0  \ge 0$ and all $x_{k_0} \in \BBR^n$, it follows that 
    $\lim_{k \to \infty} x_k = 0$.
\end{defin}

\begin{defin}
    \label{def:GUGS1}
    The zero solution of \eqref{eq:NonLinSys} is \textit{uniformly globally geometrically stable} if there exist $\alpha > 0$ and $\beta>1$ such that, 
    for all $k_0  \ge 0$ and all $x_{k_0}  \in \BBR^n$, it follows that, for all $k \ge k_0$, 
    $\| x_k \| \le \alpha \| x_{k_0} \| \beta^{-k}.$
\end{defin}

Note that, if the zero solution of \eqref{eq:NonLinSys} is uniformly globally geometrically stable, then it is uniformly globally aymptotically stable as well as uniformly Lyapunov stable.

The following three results are specializations of Theorem 13.11 given in \cite[pp. 784, 785]{Haddad2008}.
 
\begin{theo} \label{theo:LS1}
    Consider \eqref{eq:NonLinSys}, and assume  there exist a continuous function 
    $V\colon \Zplus \times \SD \to \BBR$ and 
    $\alpha_1 > 0$ such that, for all $k \ge 0$ and $x \in \SD$,
    \begin{gather}
        V(k,0)
            = 
                0, 
        \label{eq:LS_Cond1} 
        \\
        \alpha_1  \| x \| ^2
            \le
                V(k,x), 
        \label{eq:LS_Cond2} 
        \\
        V(k+1, f(k,x)) - V(k,x)
            \le
                0.
        \label{eq:LS_Cond3} 
    \end{gather}
    Then, the zero solution of \eqref{eq:NonLinSys} is {Lyapunov stable}.
\end{theo}

\begin{theo} \label{theo:ULS1}
    Consider \eqref{eq:NonLinSys}, and assume there exist a continuous function 
    $V \colon \Zplus \times \SD \to \BBR$ and 
    $\alpha_1, \beta_1 > 0$
    such that, for all $k \ge 0$ and $x \in \SD$,
    \begin{gather}
        V(k,0)
            = 
                0, \label{eq:ULS_Cond1}  \\
        \alpha_1  \| x \|^2 
            \le
                V(k,x) 
            \le
                \beta_1  \| x \| ^2 , \label{eq:ULS_Cond2}  \\
        V(k+1, f(k,x)) - V(k,x)
            \le
                0. \label{eq:ULS_Cond3} 
    \end{gather}
    Then, the zero solution of \eqref{eq:NonLinSys} is {uniformly Lyapunov stable}.
\end{theo}

\begin{theo} \label{theo:GUGS1} 
    Consider \eqref{eq:NonLinSys}, and assume there exist a continuous function 
    $V \colon \Zplus \times \BBR^n \to \BBR,$  and $\alpha_1, \beta_1, \gamma_1 > 0,$ such that, for all $k \ge 0$ and $x \in \BBR^n$,
    \begin{gather}
        \alpha_1 \| x \|^2
            \le
                V(k,x) 
            \le
                \beta_1 \| x \|^2 , \label{eq:GUGS_Cond1}  \\
        V(k+1, f(k,x)) - V(k,x)
            \le
                -\gamma_1 \| x \|^2. \label{eq:GUGS_Cond2} 
    \end{gather}
    Then, the zero solution of \eqref{eq:NonLinSys} is {uniformly globally geometrically stable}.
\end{theo}

The following result uses Theorems \ref{theo:LS1}-\ref{theo:GUGS1} to prove that, if $(\phi_k)_{k=0}^\infty$ is persistently exciting, then the RLS estimate $\theta_k$ with $\lambda \in (0,1)$ converges to $\theta$ in the sense of Definition \ref{def:GUGS1}. 
A related result is given in \cite{johnstone1982exponential}.
%

\begin{theo}
    \label{prop:RLS_stability_LP}
    Assume that $(\phi_k)_{k=0}^\infty$ is persistently exciting, 
    let $N,\alpha, \beta$ be given by Definition \ref{def:persistent_Exc}, let $R\in\BBR^{n\times n}$ be positive definite, define $P_{0} \isdef R^{-1}$, let $\lambda \in(0,1],$ and, for all $k\ge0,$ let $P_k$ be given by \eqref{eq:PUpdate}.
    Then the zero solution of \eqref{eq:theta_error1} is Lyapunov stable.
    In addition, if $\lambda \in (0,1)$, then the zero solution of \eqref{eq:theta_error1} is uniformly Lyapunov stable and uniformly globally geometrically stable. 
\end{theo}

\begin{proof}
    Define the Lyapunov candidate
    \begin{align}
        V(k, x)
            \isdef
                x ^\rmT
                P_k^{-1}
                x, \nn
    \end{align}
    where $x \in \BBR^n$. 
    Note that, for all $k\ge0$, $V(k,0) = 0$, which confirms \eqref{eq:LS_Cond1}. 
    Next, defining
    \begin{align}
        f(k,x)
            &\isdef 
                (I_n - P_{k+1} \phi_k^\rmT \phi_k) x,
        \nn 
    \end{align}
    it follows that
    \begin{align}
        V({k+1}, f(k,x)) - V({k}, x)
            &=
                f(k,x) ^\rmT P_{k+1}^{-1} f(k,x) - 
                x ^\rmT P_k^{-1} x  \nn \\
            &=
                x ^\rmT 
                [ 
                    (I_n - \phi_k^\rmT \phi_k P_{k+1} )
                    P_{k+1}^{-1} 
                    (I_n - P_{k+1} \phi_k^\rmT \phi_k)
                    -
                    P_k^{-1}
                ]
                x \nn \\
            &=
                x ^\rmT 
                [ 
                    (P_{k+1}^{-1} - \phi_k^\rmT \phi_k  )
                    (I_n - P_{k+1} \phi_k^\rmT \phi_k)
                    -
                    P_k^{-1}
                ]
                x \nn \\
            &=
                x ^\rmT 
                [ 
                    P_{k+1}^{-1} - 
                    2\phi_k^\rmT \phi_k +
                    \phi_k^\rmT \phi_k P_{k+1} \phi_k^\rmT \phi_k
                    -
                    P_k^{-1}
                ]
                x \nn \\  
            &=
                x ^\rmT 
                [ 
                     (\lambda -1 )P_{k}^{-1} 
                    -
                    \phi_k^\rmT ( I_p - \phi_k P_{k+1} \phi_{k}^\rmT) \phi_k
                ]
                x. 
        \label{eq:deltaV_le_zero}
    \end{align}
    %
    %
    
    First, consider the case where $\lambda=1$. 
    It follows from \eqref{eq:Pk_recursive} with $\lambda = 1$ that $P_0^{-1} \le P_k^{-1}$, and thus, for all $k \ge 0$,
    \begin{align}
        \sigma_{\rm min} (P_0^{-1}) \| x \|^2 \le V(k,x), \nn
    \end{align}
    which confirms \eqref{eq:LS_Cond2} with
    $\alpha_1(\| x \|) = \sigma_{\rm min} (P_0^{-1}) \| x \|^2$.
    Next, note that
    \begin{align}
        I_p - \phi_k P_{k+1} \phi_{k}^\rmT
            =
                I_p - 
                [
                \phi_k P_{k} \phi_{k}^\rmT -
                \phi_k P_{k} \phi_k ^\rmT
                    \left(
                        I_p + \phi_k  P_{k} \phi_k^\rmT
                    \right)^{-1}
                    \phi_k P_{k} \phi_{k}^\rmT
                ].
        \label{eq:Ip-phi_Pkp1_phi}
    \end{align}
    Using \eqref{eq:deltaV_le_zero}, \eqref{eq:Ip-phi_Pkp1_phi}, and Lemma \ref{Lem:Lemma_A_identity} from ``Three Useful Lemmas" yields \eqref{eq:LS_Cond3}.
    %
    It thus follows from Theorem \ref{theo:LS1} that the zero solution of \eqref{eq:theta_error1} is Lyapunov stable.

    Next, consider the case where $\lambda \in (0,1)$. 
    It follows from Proposition \ref{prop:Pkinv_bounds} that, for all $k \ge N+1$,
    \begin{align}
        \frac{\lambda^N(1-\lambda) \alpha} 
             {1 - \lambda^{N+1} }
             \| x \|^2
            \le
                V(k, x)
            &\le 
            \dfrac{\beta}{1-\lambda^{N+1}}  
                \| x \|^2
                +
                x^\rmT
                P_{N}^{-1}
                x \nn \\
            &\le
                \left(
                \dfrac{\beta}{1-\lambda^{N+1}} +
                \| P_N^{-1} \|
                \right)
                \| x \|^2,
                \nn
    \end{align}
    which confirms \eqref{eq:ULS_Cond2} for all $\lambda \in (0,1)$ with $\alpha_1  = 
    \dfrac{\lambda^N(1-\lambda) \alpha} 
             {1 - \lambda^{N+1} } $, 
    and $\beta_1  =  
    \dfrac{\beta}{1-\lambda^{N+1}} +
                \| P_N^{-1} \| .$
    Using \eqref{eq:deltaV_le_zero}, \eqref{eq:Ip-phi_Pkp1_phi}, and Lemma \ref{Lem:Lemma_A_identity} from ``Three Useful Lemmas", \eqref{eq:ULS_Cond3} is confirmed.
    It thus follows from Theorem \ref{theo:ULS1} that the zero solution of \eqref{eq:theta_error1} is uniformly Lyapunov stable.

    Furthermore, 
    \eqref{eq:GUGS_Cond1} is confirmed, 
    $\alpha_1  = 
    \frac{\lambda^N(1-\lambda) \alpha} 
             {1 - \lambda^{N+1} } $, 
    and $\beta_1  =  
    \frac{\beta}{1-\lambda^{N+1}} +
                \| P_N^{-1} \| $.
    Finally, 
    if $\lambda \in (0,1)$, then
    \begin{align}
        V({k+1}, f(k,x)) - V({k}, x)
            &\le
                (\lambda -1 ) x^\rmT P_k^{-1} x \nn \\
            &\le
                (\lambda -1 ) \left( \dfrac{\beta}{1-\lambda^{N+1}} + \| P_N^{-1} \| \right) \| x \|^2, \nn
    \end{align}
    which confirms 
    \eqref{eq:GUGS_Cond2} with , $\gamma_1 =  (1-\lambda )(\tfrac{\beta}{1-\lambda^{N+1}} + \| P_N^{-1} \|) .$
    It thus follows from Theorem \ref{theo:GUGS1} that the zero solution of \eqref{eq:theta_error1} is uniformly globally geometrically stable.
\end{proof}

The following result provides an alternative proof of Theorem \ref{prop:RLS_stability_LP} that does not depend on Theorems \ref{theo:LS1}-\ref{theo:GUGS1}.
In addition, this result considers the case
$\lambda =1$, where the RLS estimate $\theta_k$ converges to $\theta$ in the sense of Definition \ref{def:GAS1}.

\begin{theo} \label{prop:AS_GS}
    Assume that $(\phi_k)_{k=0}^\infty$ is persistently exciting, let $N,\alpha,\beta$ be given by Definition \ref{def:persistent_Exc}, 
    let $R\in\BBR^{n\times n}$ be positive definite, define $P_{0} \isdef R^{-1}$, let $\lambda \in (0,1],$ and, for all $k\ge0,$ let $P_k$ be given by \eqref{eq:PUpdate}.
    Then the zero solution of \eqref{eq:theta_error1} is globally asymptotically stable.
    Furthermore, if $\lambda \in (0,1)$, then the zero solution of \eqref{eq:theta_error1} is uniformly globally geometrically stable. 
\end{theo}

\begin{proof}
    Let $k_0 \ge 0$ and $\tilde \theta_{k_0} \in \BBR^{n}$.
    Then, it follows from \eqref{eq:theta_error_explicit} that, for all $k \ge k_0,$
    \begin{align}
        \| \tilde \theta_k \|
            &=
                \lambda^{k-k_0} \| P_{k} P_{k_0}^{-1}  \tilde \theta_{ k_0} \| 
            \nn \\
            &\le 
                \| P_{k} P_{k_0}^{-1}  \tilde \theta_{ k_0} \| 
            \nn \\
            &\le 
                \| P_{k} \|  \| P_{ k_0}^{-1} \| \|  \tilde \theta_{ k_0} \| .
        \label{eq:tilde_theta_norm}
    \end{align}
    First, consider the case where $\lambda = 1.$ %
    Let $\delta >0,$ and suppose that $\tilde \theta_{k_0}\in\BBR^n$ satisfies $\| \tilde \theta_{k_0} \| < \delta $.
    It follows from \eqref{eq:Pk_recursive} with $\lambda = 1$ that $ \| P_k \| \le \|P_0 \|$ and \eqref{eq:tilde_theta_norm}, that, for all $k \ge k_0,$  
    $
        \| \tilde \theta_k \|
            <
                \| P_{0} \|  
                \| P_{ k_0}^{-1} \| 
                \delta.
    $ 
    It thus follows from Definition \ref{def:LS1} with 
    $\varepsilon = \| P_{0} \| \| P_{ k_0}^{-1} \| \delta$
    that the zero solution of \eqref{eq:theta_error1} is Lyapunov stable.
    %

    Next, let $\tilde \theta_{ 0} \in \BBR^n$.  Then, Proposition \ref{prop:Pkinv_bounds_wo_forgetting} implies that
    \begin{align}
        \lim_{k \to \infty }\tilde \theta_{k} 
            =
        \lim_{k \to \infty } P_{k} P_{ 0}^{-1}  \tilde \theta_{ 0}
            =
            0. \nn
    \end{align}
    It thus follows from Definition \ref{def:GAS1} that the zero solution of \eqref{eq:theta_error1} is globally asymptotically stable.

    Next, consider the case where $\lambda \in (0,1)$.
    Let $k_0 \ge 0$ and $\delta >0$, and let $\tilde \theta_{k_0}\in\BBR^n$ satisfy $\| \tilde \theta_{k_0} \| < \delta $.
    It follows from Proposition \ref{prop:Pkinv_bounds} and \eqref{eq:tilde_theta_norm} that, for all $k\ge {\rm max}( N+1, k_0)$,
    \begin{align}
        \| \tilde \theta_k \|
            &<
                \varepsilon , \nn
    \end{align}
    where
    \begin{align}
        \varepsilon
            \isdef
                \dfrac
                {\beta + (1-\lambda^{N+1}) \| P_{N}^{-1} \|}
                {\lambda^N(1-\lambda)\alpha} 
                \delta
            . \nn
    \end{align}
    It thus follows from Definition \ref{def:ULS1} 
    that the zero solution of \eqref{eq:theta_error1} is uniformly Lyapunov stable.

    Next, let $\tilde \theta_{k_0} \in \BBR^n$.  Then, it follows from  \eqref{eq:theta_error_explicit} and Proposition \ref{prop:Pkinv_bounds} that, for all $\tilde \theta_{k_0} \in \BBR^n$ and $k \ge N+1$,
    \begin{align}
        \| \tilde \theta_{k} \|
            \le 
                \alpha_0 
                \| \tilde \theta_{k_0} \|
                \beta_0^{-k}, \nn
    \end{align}
    where $\beta_0 \isdef 1/\lambda$ and
    \begin{align}
        \alpha_0 
            \isdef
                \dfrac
                {\beta + (1-\lambda^{N+1}) \| P_{N}^{-1} \|}
                {\lambda^N(1-\lambda)\alpha}. \nn
    \end{align}
    It thus follows from Definition \ref{def:GUGS1} that the zero solution of \eqref{eq:theta_error1} is uniformly globally geometrically stable, and thus globally asymptotically stable.
    %
    %
    %
\end{proof}


The following result shows that persistent excitation produces an infinite sequence of matrices whose product converges to zero.

\begin{prop}
\label{prop:Ak_PE}
Let $P_0\in\BBR^{n\times n}$ be positive definite, let $\lambda \in(0,1],$ and, for all $k\ge0,$ let $P_k$ be given by \eqref{eq:PUpdate}.
Then, for all $k\ge0,$ all of the eigenvalues of $P_{k+1} \phi_k^\rmT \phi_k$ are contained in $[0,1]$.
%
If, in addition, $\SeqPhi$ is persistently exciting, then
\begin{align}
    \lim_{k\to\infty} \SA_k = 0, 
    \label{eq:SA_k}
\end{align}
where
\begin{align}
    \SA_k
        \isdef
            (I_n - P_{k+1} \phi_k^\rmT \phi_k)\cdots (I_n - P_{1} \phi_0^\rmT \phi_0). 
\end{align}

\end{prop}

\begin{proof}
    It follows from \eqref{eq:Pk_recursive} that, for all $k\ge0,$ $\phi_k^\rmT \phi_k \le P_{k+1}^{-1}$, and thus, for all $k\ge0,$ 
    $  P_{k+1}^{1/2} \phi_k^\rmT \phi_k P_{k+1}^{1/2} \le I_n$.
    Hence, for all $k\ge0,$  $$0\le\lambda_{\rm max} (P_{k+1}  \phi_k^\rmT \phi_k  )=\lambda_{\rm max} (P_{k+1}^{1/2} \phi_k^\rmT \phi_k P_{k+1}^{1/2}) \le1.$$
    To prove \eqref{eq:SA_k}, suppose that
    $\SeqPhi$ is persistently exciting, let $i\in\{1,\ldots,n\},$ and define $\theta_{0}\isdef e_i+\theta,$ where $e_i$ is the $i$th column of $I_n.$
    Note that $\tilde\theta_{0}\isdef  \theta_{0}-\theta = e_i.$
    Then, \eqref{eq:theta_error_explicit} implies that, for all $k\ge0$,  
    \begin{align}
        \tilde \theta_{k+1} 
            =
                \SA_k e_i
            = 
                 \lambda^{k+1} P_{k+1} P_{0}^{-1} e_i.\label{SAi} 
    \end{align}
    It follows from Theorem \ref{prop:AS_GS} that $\tilde \theta_k$ converges to zero. 
   Hence, \eqref{SAi} implies that the $i$th column of $\SA_k$ converges to zero as $k\to\infty.$
    It thus follows that every column of $\SA_k$ converges to zero as $k\to\infty,$
    which implies \eqref{eq:SA_k}. 
    %
    %
\end{proof}

It follows from Theorem \ref{prop:AS_GS} that, if $\SeqPhi$ is persistently exciting, then, for all $\lambda \in (0,1]$, $\tilde \theta_k$ converges to zero. 
In addition, if $\lambda \in (0,1)$, then $\tilde \theta_k$ converges to zero geometrically, and thus the rate of convergence of $\| \tilde \theta_k \|$ is $O(\lambda^k)$. 
However, in the case $\lambda = 1,$ as shown in \cite{johnstone1982exponential} and the next example, $\tilde \theta_k$ converges to zero as $O(1/k),$ and thus the convergence is not  geometric.

\begin{exam}
\label{exam:ConvRate_lambda}
\textit{Effect of $\lambda$ on the rate of convergence of $\theta_k$.}
Consider the 3rd-order FIR system 
\begin{align}
    y_k
        =
            \frac{\shiftq^2 + 0.8 \shiftq + 0.5}{\shiftq^3}
            u_k.
\end{align}
%
%
To apply RLS, let $\theta = [1\ 0.8\ 0.5]$, $\theta_0 = 0$, and $\phi_k =  [u_{k-1} \ u_{k-2} \ u_{k-3}]$, where the input $u_k$ is zero-mean Gaussian white noise with standard deviation 1.
Note that $\SeqPhi$ is persistently exciting. 
It thus follows from Theorem \ref{prop:AS_GS} that $\tilde \theta_k$ converges to zero. 
Figure \ref{fig:CSM_forgetting_lambda_array} shows the parameter-error norm $\| \tilde \theta_{k} \|  $ for several values of $P_0$ and $\lambda$ as well as the condition number of the corresponding $P_k$.
Note that the convergence rate of $\| \tilde \theta_k \|$ is $O(1/k)$ for $\lambda =1$ and geometric for all $\lambda \in (0,1)$. 
Furthermore, as $\lambda$ is decreased, the convergence rate of $ \theta_k$ increases; however, the condition number of $P_k$ degrades, and the effect of $P_0$ is reduced.
\EndExample
\end{exam}

\begin{figure}[h!]
    \centering
    \includegraphics[width=.9 \textwidth]{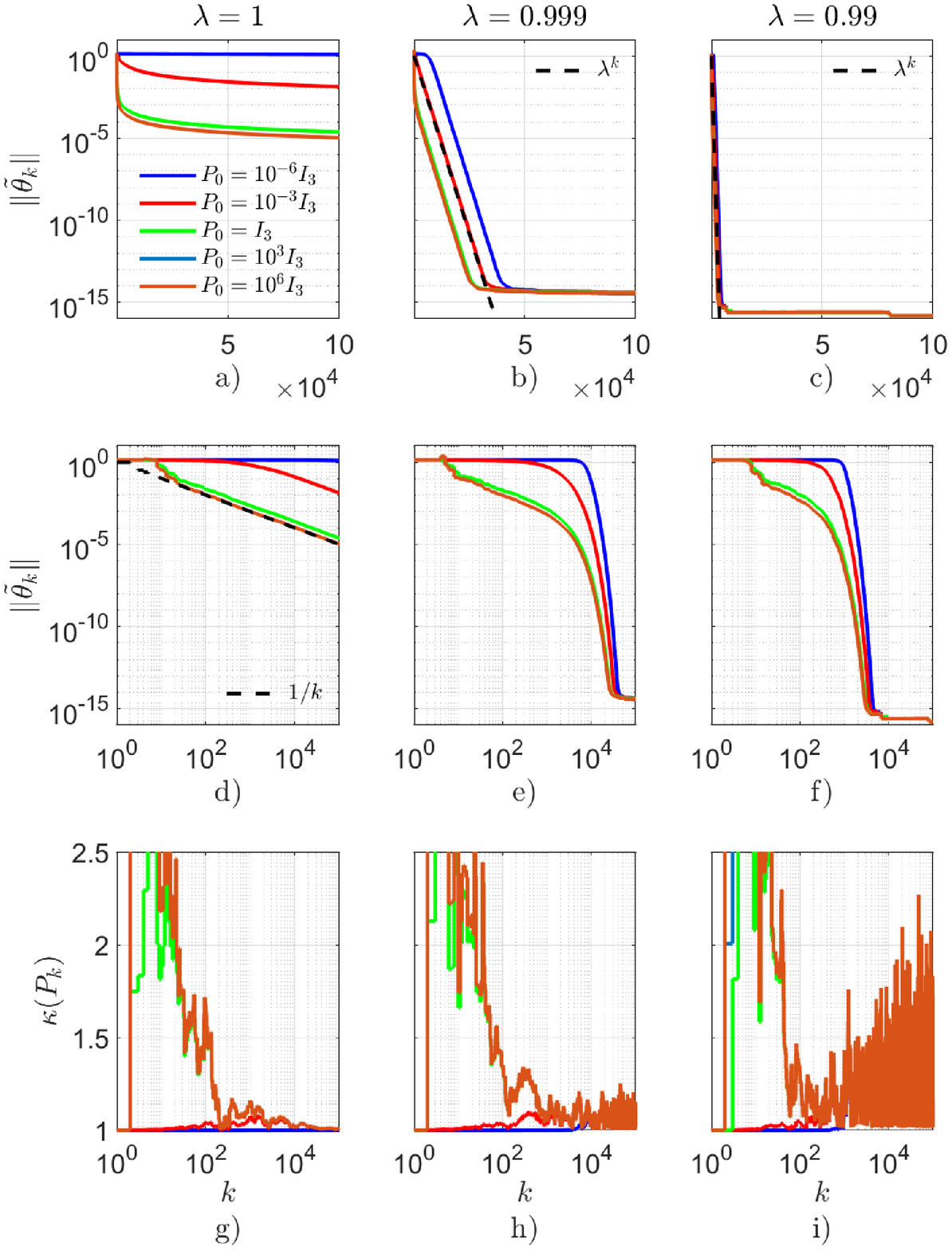}
    \vspace{-3ex}
    \caption
        {
            Example \ref{exam:ConvRate_lambda}.
            Effect of $\lambda$ on the rate of convergence of $\theta_k$. 
            a)-f) show the parameter error norm $\| \tilde \theta_{k} \|  $ for several values of $P_0$ and $\lambda$.
            Note that the slope of $-1$ between $\log \| \tilde \theta_k \|$ and $\log k$ in d) is consistent with the fact that the rate of convergence of $\| \tilde \theta_k \|$ is $O(1/k)$ for $\lambda = 1$.
            Similarly, the slope of $\log \lambda$ between $\log \| \tilde \theta_k \|$ and $k$ in b) and c) is consistent with the fact that the rate of convergence of $\| \tilde \theta_k \|$ is $O(\lambda^k)$ for $\lambda \in (0,1)$.
            g), h), and i) show the condition number of the corresponding $P_k$ for several values of $P_0$ and $\lambda$.
            Note that, as $\lambda$ is decreased, the convergence rate of $ \theta_k$ increases; however, the condition number of $P_k$ degrades, and the effect of $P_0$ is reduced.
        }
    \label{fig:CSM_forgetting_lambda_array}
\end{figure}

\section{Lack of Persistent Excitation}

This section presents numerical examples to investigate the effect of lack of persistent excitation. 
As shown in Example \ref{exam:PE_Pk_bounds_2} and Example \ref{exam:persistency_conditionNumber},
if $\SeqPhi$ is not persistently exciting and $\lambda = 1$, then some of the singular values of $P_k$ converge to zero, whereas the remaining singular values remain bounded. 
On the other hand, if $\SeqPhi$ is not persistently exciting and $\lambda \in (0, 1)$, then some of the singular values of $P_k$ remain bounded, whereas the remaining singular values diverge.  
%
Furthermore, Proposition \ref{prop:z_converges} implies that the predicted error $z_k$ 
converges to zero whether or not $\SeqPhi$ is persistent.

\begin{exam}
    \label{exam:ScalarEstimation}
    \textit{Lack of persistent excitation in scalar estimation}.
    Let $n=1$, so that \eqref{eq:PUpdate}, \eqref{eq:thetaUpdate} are given by
    \begin{align}
            P_{k+1}
                &=
                    \dfrac{P_{k}}{\lambda + P_{k} \phi_k ^2} ,
            \label{eq:P_scalar}
                    \\
            \tilde \theta_{k+1}
                &=
                    \dfrac{\lambda \tilde \theta_{k}}{\lambda + P_{k} \phi_k ^2}.
            \label{eq:theta_scalar}
    \end{align} 
    Now, let $k_0\ge0$ and assume that, for all $k \ge k_0,$  $\phi_k = 0.$ 
    Therefore, for all $j\ge0$ and $N\ge1,$
    $F_{j,j+N}$ cannot be lower bounded as in \eqref{eq:persistent_def}, and thus
    $\SeqPhi$ is not persistently exciting.
    %
    %
    %
    Furthermore, in the case where $\lambda = 1$, it follows from the fact that $\phi_k = 0$ for all $k\ge k_0$  that $P_k$ and $\tilde \theta_k$ converge in $k_0$ steps to $\overline P\ne0$ and $\overline {\tilde \theta}$, respectively.
    Furthermore, if $\theta_0 \ne \theta,$ then $\overline {\tilde \theta}\ne0.$
    %
    %
    %
    %
    However, in the case where $\lambda \in (0, 1)$, it follows that $P_k$ diverges geometrically, whereas, as in the case where $\lambda = 1,$ $\tilde \theta_k$ converges in $k_0$ steps.
     Therefore, for all $\lambda\in(0,1],$
    since $\phi_k=0$ for all $k \ge k_0,$ it follows from \eqref{eq:P_scalar} and \eqref{eq:theta_scalar} that, for all $k \ge k_0,$ the minimum value of \eqref{eq:J_LS} is achieved in a finite number of steps.
    Consequently, RLS provides no further refinement of the estimate $\theta_k$ of $\theta,$ and thus $\overline {\tilde \theta}\ne0$ implies that $\theta_k$ does not converge to $\theta.$

    Alternatively, assume that, for all $k\ge0,$ $\phi_k = \overline \phi,$ where  $\overline \phi\neq 0$. 
    Then it follows from Definition \ref{def:persistent_Exc} with $N=1$, $\alpha = \overline{\phi}^2$, and $\beta  = 3\overline{\phi}^2$ that $\SeqPhi$ is persistently exciting. 
    If $\lambda = 1$, then both $P_k$ and $\tilde \theta_k$ converge to zero.
    However, if $\lambda \in (0, 1)$, then $P_k$ converges to $\tfrac{1-\lambda}{ \overline \phi^2}$ and $\tilde \theta_k$ converges geometrically to zero. 
    Table \ref{tab:RLS_PE_lambda} shows the asymptotic behavior of $\tilde \theta_k $ and $ P_k$ for both of these cases. 
    \EndExample
\end{exam}

\begin{table}[h!]
    \centering
    \begin{tabular}{|l|c|c|}
        \hline
        Excitation $\backslash$ $\lambda$  & $\lambda = 1$ & $\lambda \in (0,1)$   \\
        \hline
        Not persistently exciting  &    $\tilde \theta_k \to \overline {\tilde \theta}, P_k \to \overline P$  &    $\tilde \theta_k \to \overline {\tilde \theta}, P_k$ diverges  \\
        \hline
        Persistently exciting  &    $\tilde \theta_k\to0, P_k  \to 0$    &  $\tilde \theta_k \to 0, P_k \to \tfrac{1-\lambda}{ \overline \phi^2}$ \\
        \hline
    \end{tabular}
    \vspace{1ex}
    \centering
    \caption{Asymptotic behavior of RLS in Example \ref{exam:ScalarEstimation}.  
    In the case of persistent excitation with $\lambda<1,$ the convergence of $\tilde\theta_k$ is geometric.}
    \label{tab:RLS_PE_lambda}
\end{table}

\begin{exam}\label{exam:LoP_n2}
    \textit{Subspace-constrained regressor.}
    Consider  \eqref{eq:process}, where $\phi_k = (\sin \tfrac{2\pi k}{100}) [1 \ 1]$ and $\theta = [0.4 \ 1.4]^\rmT$.
    To estimate $\theta$ using RLS, let $P_0 = I_2$ and $\theta_0 = 0$.
    Figure \ref{fig:CSM_forgetting_LoP_n_2_100} shows the estimate $\theta_k$ of $\theta$ with $\lambda = 1$ and $\lambda = 0.99$.
    Note that all regressors $\phi_k$ lie along the same one-dimensional subspace, and thus, $\SeqPhi$ is not persistently exciting. 
    It follows from \eqref{eq:theta_in_range_PHI} that the estimate $\theta_k$ of $\theta$ lies in this subspace.
    For $\lambda = 1$, note that  one singular value decreases to zero, whereas the other singular value is bounded. 
    Note that $\tilde \theta_k$ converges along the singular vector corresponding to the bounded singular value. 
    For $\lambda = 0.99,$ one singular value is bounded, whereas the other singular value diverges. %
    Note that $\tilde \theta_k$ converges along the singular vector corresponding to the diverging singular value. 
    \EndExample    
\end{exam}

\begin{figure}[h!]
    \centering 
    \begin{subfigure}[b]{1\textwidth}
        \centering
        \includegraphics[width=0.6\textwidth]{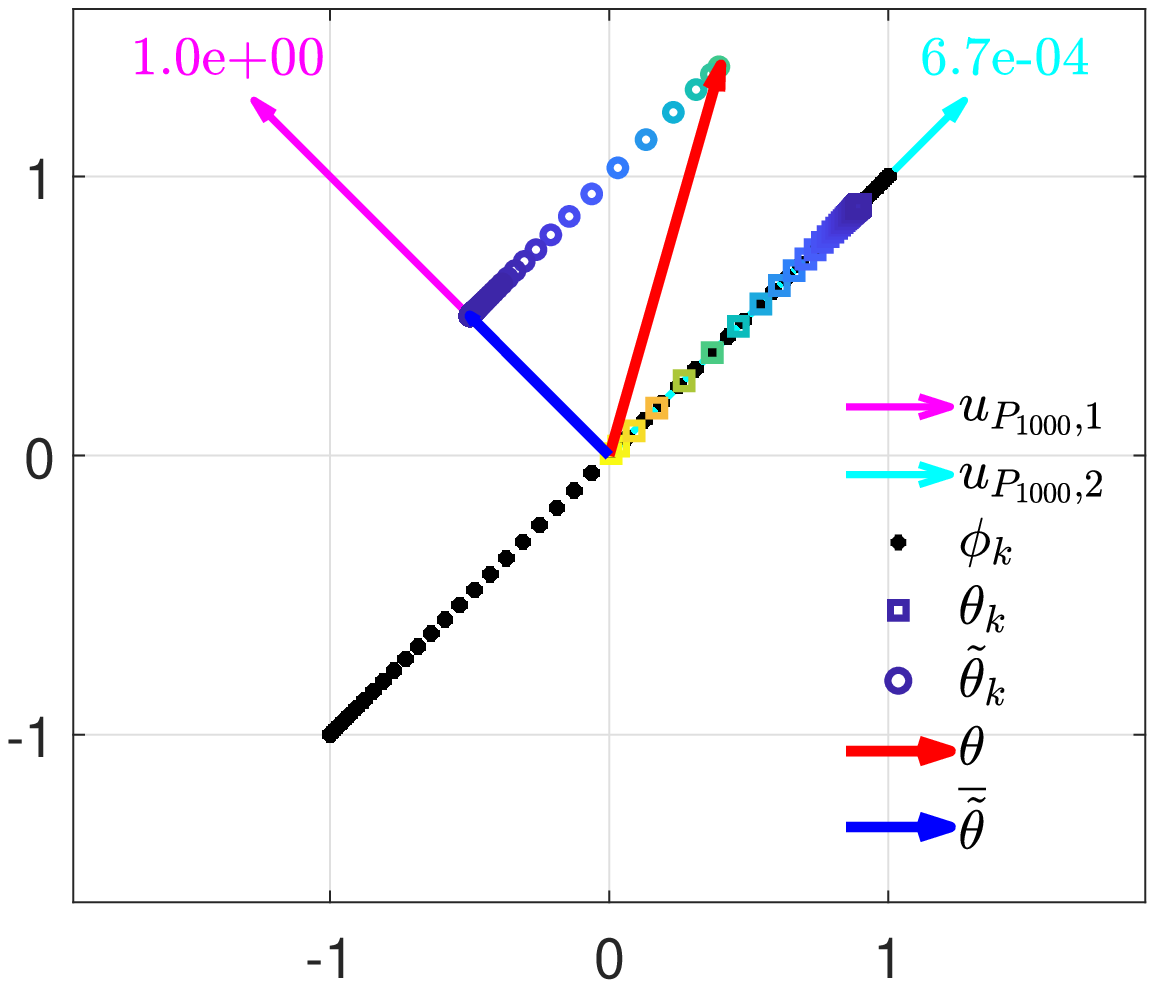}
        \caption*{$\lambda = 1$}
        \label{Fig.CSM_forgetting_LoP_n_2_100}
    \end{subfigure}
    
    \begin{subfigure}[b]{1\textwidth}
        \centering
        \includegraphics[width=0.6\textwidth]{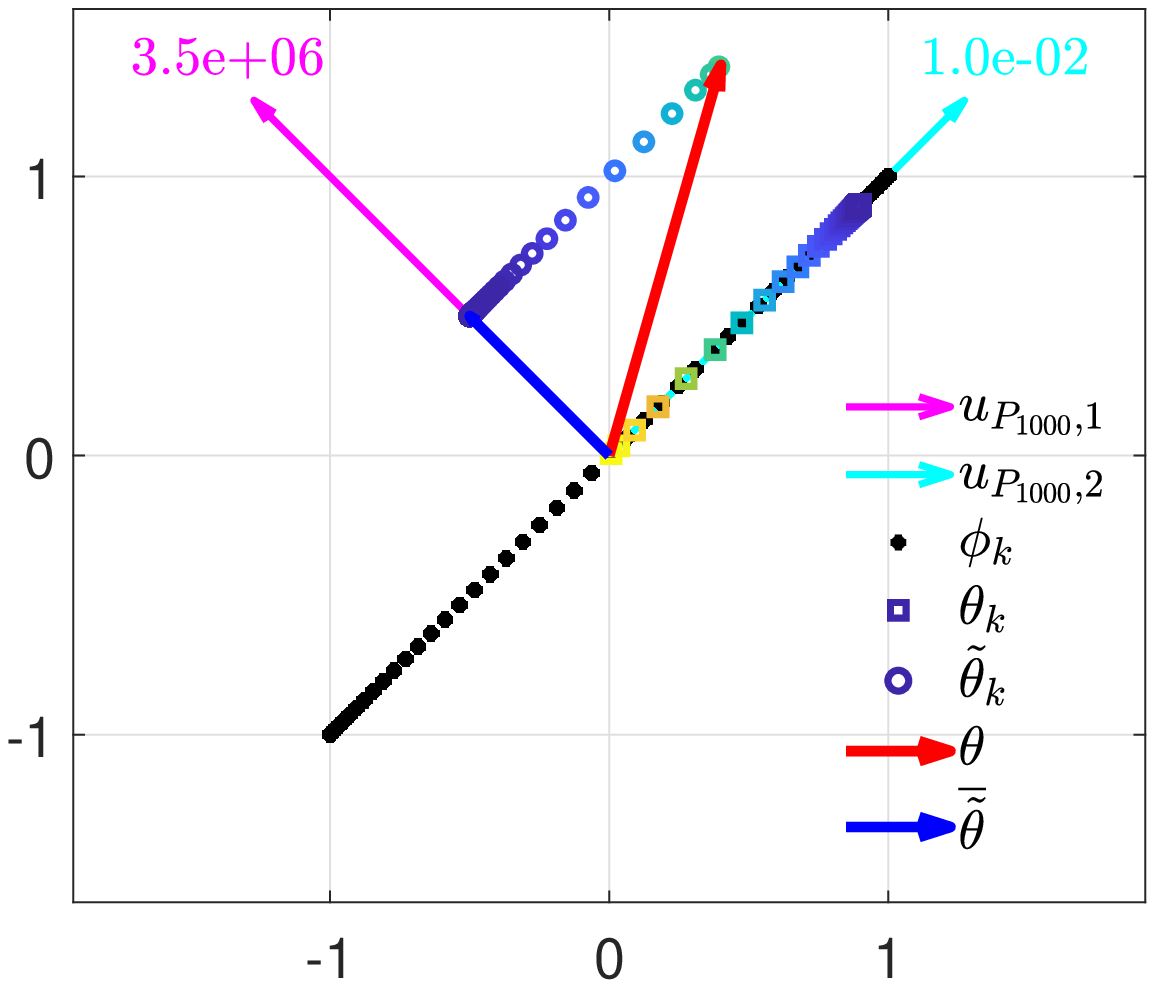}
        \caption*{$\lambda = 0.99$}
        \label{Fig.CSM_forgetting_LoP_n_2_99}
    \end{subfigure}
    \caption
        {
            Example \ref{exam:LoP_n2}. 
            Subspace constrained regressor.
            The first component of each vector is plotted along the horizontal axis, and the second component is plotted along the vertical axis. 
            The singular values $\sigma_i(P_{1000})$ are shown with the corresponding singular vector $u_{P_{1000},i}$.
            All regressors $\phi_k$ lie along the same one-dimensional subspace, and thus, $\SeqPhi$ is not persistently exciting. 
            Consequently, each estimate $\theta_k$ of $\theta$ lies in this subspace.
            The color gradient from yellow to blue of $\theta_k$ and $\tilde \theta_k$ shows the evolution from $k=1$ to $k=1000$.
            In a), the singular value corresponding to the cyan singular vector decreases to zero, whereas the singular value corresponding to the magenta singular vector is bounded. 
            Note that $\tilde \theta_k$ converges along the singular vector corresponding to the bounded singular value. 
            In b), the singular value corresponding to the cyan singular vector is bounded, whereas the singular value corresponding to the magenta singular vector diverges. 
            Note that $\tilde \theta_k$ converges along the singular vector corresponding to the diverging singular value.
        }
    \label{fig:CSM_forgetting_LoP_n_2_100}
\end{figure}

\begin{exam}
\label{exam:LackofPE_theta_k}
\textit{Lack of persistent excitation and finite-precision arithmetic.}
Consider the problem of fitting a 5th-order model to measured input-output data from the system \eqref{eq:5thOrderG}, where the input $u_k$ is given by \eqref{eq:u_harm}.
Note that $\phi_k$ is given by \eqref{eq:Reg_IIR}, and is not persistently exciting as shown in Example \ref{exam:persistency_conditionNumber}.
Let $P_0 = I_{10}$, $\theta_0 = 0$, and $\lambda = 0.999$. 
Figure \ref{fig:CSM_forgetting_5thOrder_fit_Standard} shows the predicted error $z_k $, 
the norm of the parameter error $\tilde \theta_k$, and 
the singular values and the condition number of $P_k$.
Note that the $\tilde \theta_k$ does not converge to zero and that six singular values of $P_k$ remain bounded due to the presence of three harmonics in the regresssor.
Due to finite-precision arithmetic, the computation becomes erroneous as $P_k$ becomes numerically ill-conditioned, and thus the estimate $\theta_k$ diverges. 
\EndExample
\end{exam}

\begin{figure}[h!]
    \centering
    \includegraphics[width=1 \textwidth]{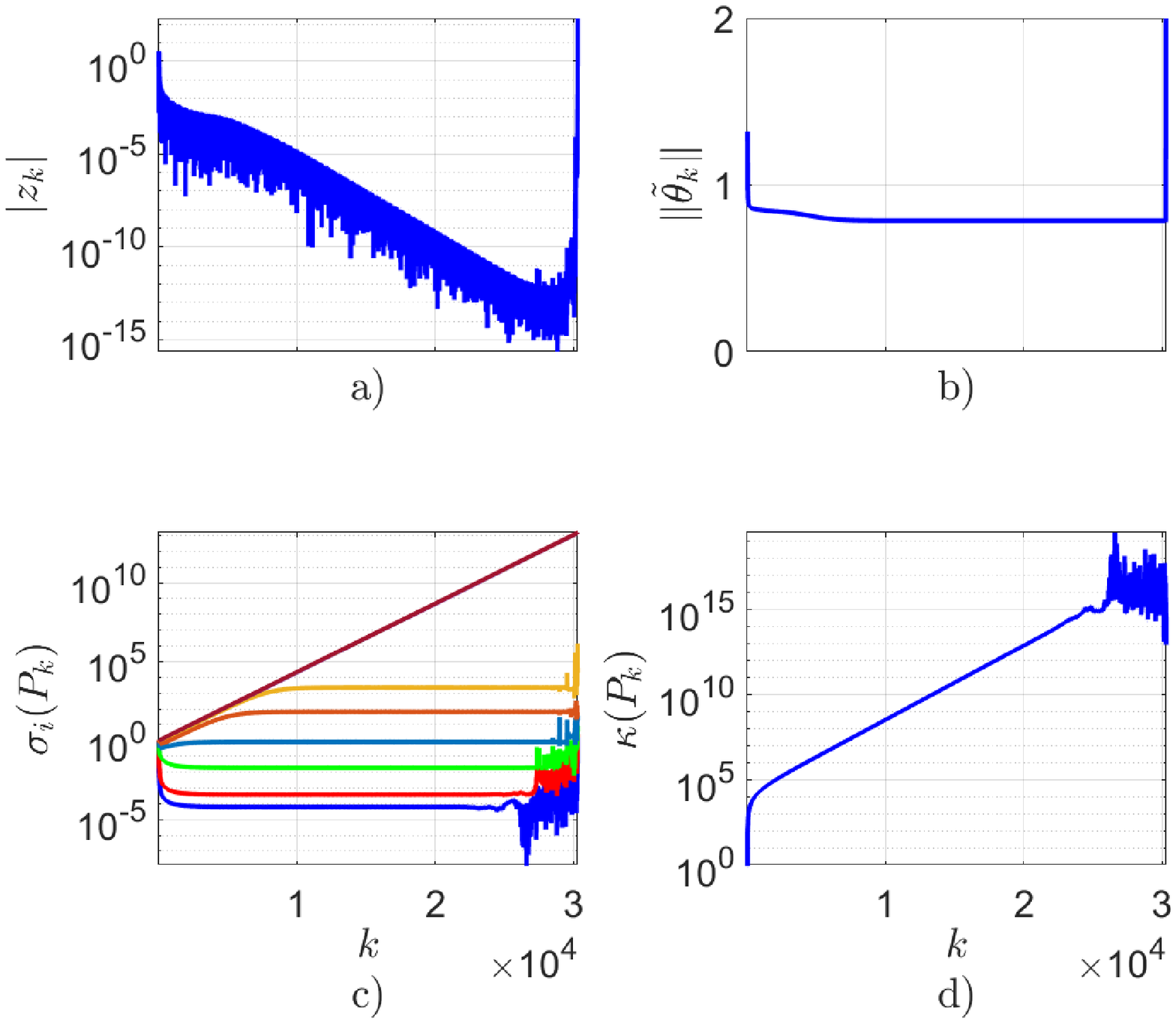}
    \caption{
    Example \ref{exam:LackofPE_theta_k}. Effect of lack of persistent excitation on $\theta_k$. 
    a) shows the predicted error $z_k$,  
    b) shows the norm of the parameter error $\tilde \theta_k$, 
    c) shows the singular values of $P_k$, and
    d) shows the condition number of $P_k$.
    Note that six singular values of $P_k$ remain bounded due to the presence of three harmonics in the regresssor.
    Due to finite-precision arithmetic, the computation becomes erroneous as $P_k$ becomes numerically ill-conditioned, and thus, the estimate $\theta_k$ diverges.
    }
    \label{fig:CSM_forgetting_5thOrder_fit_Standard}
\end{figure}

The numerical examples in this section show that, if $\lambda \in (0,1]$ and $\SeqPhi$ is not persistently exciting, then $\tilde \theta_k$ does not necessarily converge to zero. 
Furthermore, if $\lambda\in(0,1)$ and $\SeqPhi$ is not persistently exciting, then some of the singular values of $P_k$ diverge, and $\theta_k$ diverges due to finite-precision arithmetic when $P_k$ becomes numerically ill-conditioned. 
%

\section{Information Subspace}
Using the singular value decomposition,   \eqref{eq:Pk_recursive} can be written as
\begin{align}
    P_{k+1}^{-1}
        =
            \lambda 
            U_{k} \Sigma_{k} U_{k}^\rmT
            +
            U_{k} \psi_k^\rmT \psi_k U_{k}^\rmT
            ,
    \label{eq:Atheta_EF_SVD}
\end{align}
where $U_{k}\in \BBR^{n \times n}$ is an orthonormal matrix whose columns are the singular vectors of $P_{k}^{-1}$, 
$\Sigma_{k}\in \BBR^{n \times n}$ is a diagonal matrix whose  diagonal entries  are the corresponding singular values, and  
\begin{align}
    \psi_k 
        \isdef
            \phi_k U_{k}.
    \label{eq:psi_k_definition}
\end{align}
The columns of $U_k$ are the
\textit{information directions} at step $k$, and each row of $\psi_k$ is the projection of the corresponding row of $\phi_k$ onto the information directions. 
The norm of each column of $\psi_k$ thus indicates the \textit{information content} present in $\phi_k$ along the corresponding information direction.
The smallest subspace that is spanned by a subset of the information directions and that contains all rows of $\phi_k$  is the \textit{information-rich subspace} $\SI_k$ at step $k$.
Figure \ref{fig:CSM_Info_subspace} illustrates the information-rich subspace.

\begin{figure}[h!]
    \centering
    \includegraphics[width=1 \textwidth]{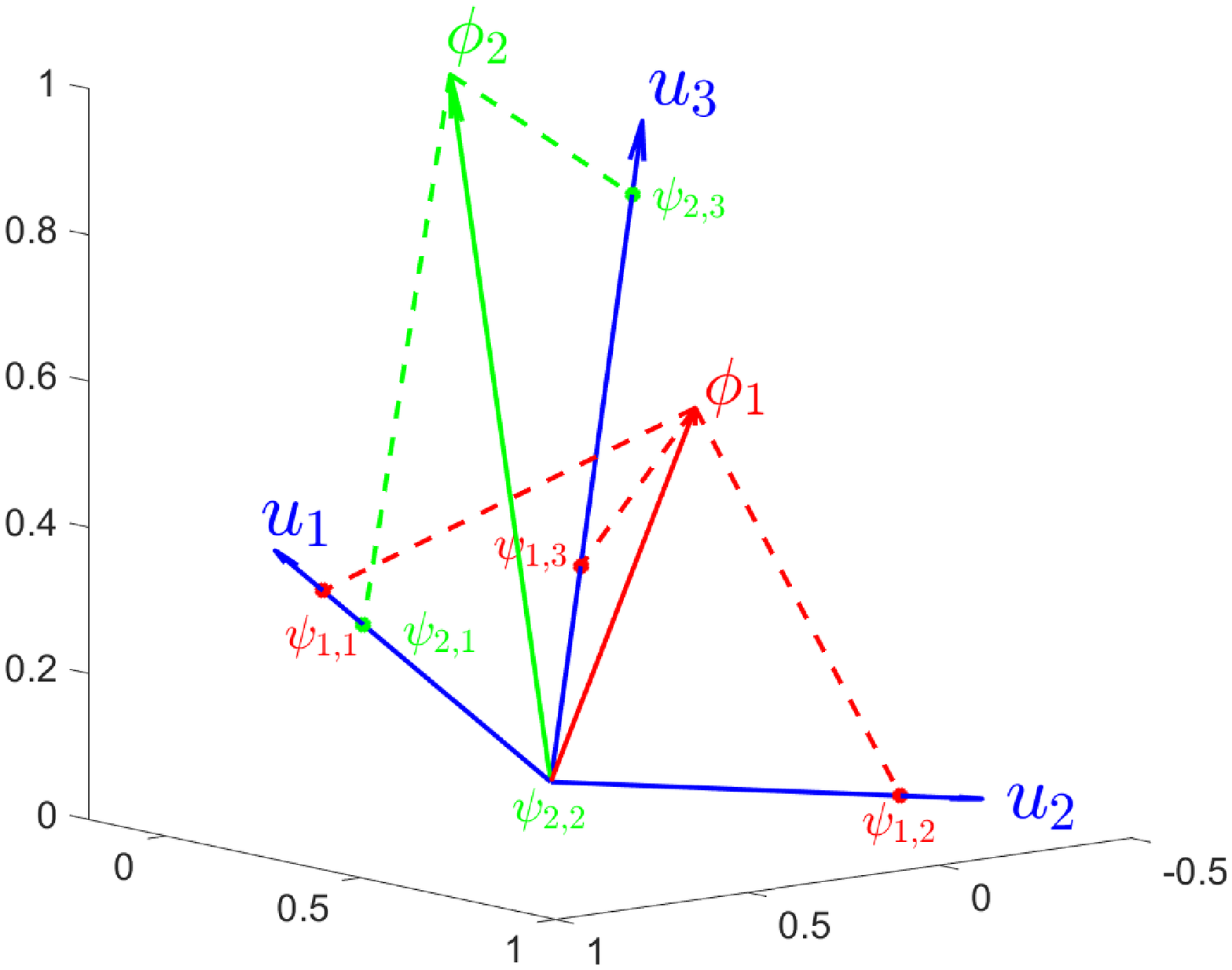}
    \caption{
    Illustrative example of the information-rich subspace. 
    Let $u_1, u_2$, and $u_3$ be the information directions (shown in blue). 
    The regressor $\phi_1$ (shown in red) has new information along all three information directions, as shown by the nonzero values  $\psi_{1,1},$ $\psi_{1,2},$ and $\psi_{1,3}$; the information-rich subspace is thus $\SR([u_1 \ u_2 \ u_3])$.
    On the other hand, the regressor $\phi_2$ (shown in green) has new information only along $u_1$ and $u_3$, as shown by the nonzero values $\psi_{2,1}$ and $\psi_{2,3}$;  the information-rich subspace is thus $\SR([u_1 \ u_3])$.
    }
    \label{fig:CSM_Info_subspace}
\end{figure}

Now, consider the case where
\begin{align}
    \psi_k
        =
            \matl{cc}
                \psi_{k,1} & 
                0_{p \times (n-n_1)}
            \matr,
    \label{eq:psi_k_0}
\end{align}
where $\psi_{k,1} \in \BBR^{p \times n_1}$.
It follows from \eqref{eq:psi_k_0} that $\phi_k$ provides new information along the first $n_1$ columns of $U_k$; these directions constitute the information-rich subspace.
It thus follows from \eqref{eq:Atheta_EF_SVD} and \eqref{eq:psi_k_0} that $P_{k+1}^{-1}$ is given by
\begin{align}
    P_{k+1}^{-1}
        =
            U_{k} 
            \matl{cc}
                \lambda \Sigma_{k,1} + \psi_{k,1}^\rmT \psi_{k,1}  & 0 \\
                0               & \lambda \Sigma_{k,2}
            \matr
            U_{k}^\rmT,
    \label{eq:Atheta_EF_SVD_Analysis}
\end{align}
where $\Sigma_{k,1} \in \BBR^{n_1 \times n_1}$ is the diagonal matrix whose diagonal entries are the first $n_1$  singular values of $P_{k}^{-1}$, and $\Sigma_{k,2}$ is the diagonal matrix whose diagonal entries are the remaining $n-n_1$ singular values of $P_{k}^{-1}$.
In particular, writing 
\begin{align}
    U_k
        = 
            \matl{cc}
                U_{k,1} & U_{k,2}
            \matr,
\end{align}
where $U_{k,1} \in \BBR^{n \times n_1}$ contains the first $n_1$ columns of $U_k$, and $U_{k,2} \in \BBR^{n \times n-n_1}$ contains the remaining $n - n_1$ columns of $U_k$, it follows that
\begin{align}
    P_{k+1}^{-1}
        =
            \matl{cc}
                U_{k+1,1} & U_{k+1,2}
            \matr
            \matl{cc}
                \Sigma_{k+1,1}  & 0 \\
                0               & \Sigma_{k+1,2}
            \matr
            \matl{c}
                U_{k+1,1}^\rmT \\ U_{k+1,2}^\rmT
            \matr,
    \label{eq:Pkp1inv_SVD}
\end{align}
where
\begin{align}
    U_{k+1,1} 
        &=
            U_{k,1} V_k
    \label{eq:U1kp1}
    , \\
    \Sigma_{k+1,1}
        &=
            D_k, \\
    U_{k+1,2} 
        &=
            U_{k,2}
    \label{eq:U2kp1}
    , \\
    \Sigma_{k+1,2}
        &=
            \lambda \Sigma_{k,2},
    \label{eq:Sigma2kp1}
\end{align}
where
$V_k \in \BBR^{n_1 \times n_1}$ contains the singular vectors of $\lambda \Sigma_{k,1} + \psi_{k,1}^\rmT \psi_{k,1} $ and $D_k \in \BBR^{n_1 \times n_1}$ is the diagonal matrix containing the corresponding singular values. 
It follows from \eqref{eq:U2kp1}, \eqref{eq:Sigma2kp1} that if, for all $k \ge 0$, $\psi_k$ is  given by \eqref{eq:psi_k_0} and $\lambda \in (0,1)$, then the last $n-n_1$ singular vectors of $P_k^{-1}$ do not change and the corresponding singular values of $P_k^{-1}$ decrease to zero geometrically.
It thus follows from Proposition \ref{prop:Pkinv_bounds} that $\SeqPhi$ is not persistently exciting. 
Furthermore, since $P_k$ and $P_k^{-1}$ have the same singular vectors and the singular values of $P_k$ are the reciprocals of the singular values of $P_k^{-1}$, it follows that the last $n-n_1$ singular values of $P_k$ diverge.

The next example considers the case where there exists a proper subspace $\SSS\subset\BBR^n$ such that, for all $k\ge0,$ $\SR(\phi_k^\rmT)\subseteq\SSS$.
Hence, $\SeqPhi$ is not persistently exciting.
In this case, for all $k\ge0,$ the information-rich subspace $\SI_k$ is a proper subspace of $\BBR^n,$ and the singular values of $P_k^{-1}$ corresponding to the singular vectors in the orthogonal complement of $\SI_k$ converge to zero.

\begin{exam}
    \label{exam:Persistency_InformationContent}
    \textit{Lack of persistent excitation and the information-rich subspace.}
    Consider the regressor $\phi_k$ given by \eqref{eq:Reg_IIR} used in Example \ref{exam:persistency_conditionNumber}. 
    Recall that $\SeqPhi$ is not persistently exciting. %
    Let $P_0 = I_{10}$.
    Figure \ref{fig:CSM_forgetting_Persistency_InformationContent} shows the information content 
    $| \psi_{k,(i)} | $ for several values of $\lambda$ along with the singular values of the corresponding $P_k^{-1}$.
    Note that the information-rich subspace is six dimensional due to the presence of three harmonics in $u_k$ as shown by six relatively large components of $\psi_k$ and, in the case where $\lambda<1$, the singular values that correspond to the singular vectors not in the information-rich subspace converge to zero in machine precision. 
    \EndExample
\end{exam}

\begin{figure}[h!]
    \centering
    \includegraphics[width=1 \textwidth]{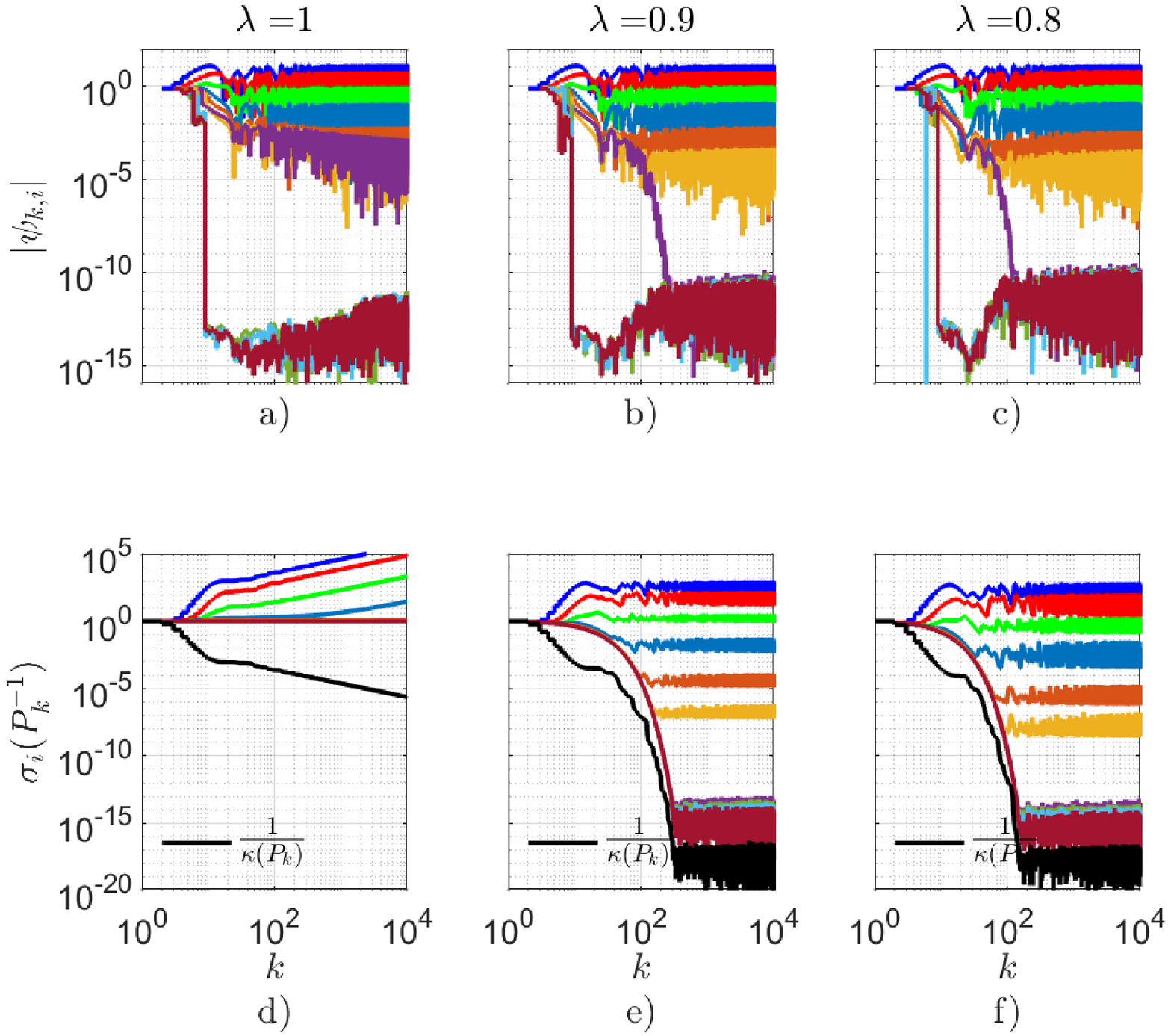}
    \caption{
    Example \ref{exam:Persistency_InformationContent}. 
    Relation between $P_k$ and the information content $\psi_k$.
    a), b), and c) show the information content $ {\rm col}_i (\psi_{k}) $ for several values of $\lambda$.
    Note that, in each case, the information-rich subspace is six dimensional due to the presence of three harmonics in $u_k$. 
    d), e), and (f) show the singular values of $P_k^{-1}$ for several values of $\lambda$.
    The inverse of the condition number of $P_k$ is shown in black. 
    Note that, for $\lambda<1$, the singular values of $P_k^{-1}$ corresponding to the singular vectors in the orthogonal complement of the information-rich subspace converge to zero.
    }
    \label{fig:CSM_forgetting_Persistency_InformationContent}
\end{figure}

\section{Variable-Direction forgetting}

Examples \ref{exam:PE_Pk_bounds_2},  \ref{exam:persistency_conditionNumber}, \ref{exam:ScalarEstimation}, \ref{exam:LoP_n2}, and \ref{exam:LackofPE_theta_k} show that some of the singular values of $P_k^{-1}$ converge to zero in the case where $\phi_k$ is not persistently exciting.
To address this situation, \eqref{eq:Pk_recursive} is modified by replacing the scalar forgetting factor $\lambda$ by a data-dependent forgetting matrix $\Lambda_k$.
Similar modifications are discussed in ``Toward Matrix Forgetting''.
In particular, $P_{k+1}^{-1}$ is redefined as
%
%
%
%
\begin{align}
    P_{k+1}^{-1}
        =
            \Lambda_k  P_{k}^{-1} \Lambda_k 
            +
            \phi_k ^{\rmT}\phi_k,
    \label{eq:Atheta_TF}
\end{align}
where $\Lambda_k$ is a positive-definite (and thus symmetric) matrix constructed below. 
Note that, for all $k\ge0,$  $P_{k+1}^{-1}$ given by \eqref{eq:Atheta_TF} is  positive definite. 
Using the singular value decomposition, \eqref{eq:Atheta_TF} can be written as
\begin{align}
    P_{k+1}^{-1}
        =
            \Lambda_k 
            U_{k} \Sigma_{k} U_{k}^\rmT
            \Lambda_k 
            +
            U_{k} 
            \psi_k^\rmT \psi_k
            U_{k}^\rmT,
    \label{eq:Atheta_TF_SVD}
\end{align}
where $U_{k}$, $\Sigma_{k}$, and $\psi_{k}$ are as defined in the previous section. 

The objective is to apply forgetting to only those singular values of $P_k^{-1}$ that correspond to the singular vectors in the information-rich subspace, that is, forgetting is restricted to the subspace of $P_k^{-1}$ where sufficient new information is provided by $\phi_k$.
Specifically, forgetting is applied to those information directions where the information content is greater than $\varepsilon > 0$, where $\varepsilon$ should be selected to be larger than the noise to signal ratio or larger than the machine zero, if no noise is present.
To do so, \eqref{eq:Atheta_TF_SVD} is written as 
\begin{align}
    P_{k+1}^{-1}
        =
            U_{k} \overline \Lambda_k 
            \Sigma_{k} 
            \overline \Lambda_k U_{k}^\rmT
            +
            U_{k} 
            \psi_k^\rmT \psi_k
            U_{k}^\rmT,
    \label{eq:Atheta_TF_SVD_flipped}
\end{align}
where $\overline \Lambda_k$ is a diagonal matrix whose diagonal entries are either $\sqrt \lambda$ or $1$.
In particular, 
\begin{align}
    \overline \Lambda_k(i,i)
        \isdef 
            \begin{cases}
                \sqrt \lambda,    & \| {\rm col}_i (\psi_{k}) \| > \varepsilon, \\
                1,    & \rm{otherwise},
            \end{cases}
    \label{eq:LambdaBar_def}
\end{align}
where ${\rm col}_i (\psi_{k})$ is the $i$th column of $\psi_k$ and
$\lambda \in (0,1]$.
Note that, it follows from \eqref{eq:Atheta_TF_SVD_flipped} and \eqref{eq:LambdaBar_def} that $P_{k+1}\inv$ is positive definite. 
Next, it follows from \eqref{eq:Atheta_TF_SVD} and \eqref{eq:Atheta_TF_SVD_flipped} that 
\begin{align}
    \Lambda_k 
        =
            U_{k} \overline \Lambda_k U_{k}  ^\rmT,
    \label{eq:Lambda_k_def}
\end{align}
which is positive definite.
Note that
\begin{align}
    \Lambda_k^{-1}
        &=
            U_{k} \overline \Lambda_k^{-1} U_{k}  ^\rmT.
    \label{eq:LambdaInv_def}
\end{align}

The next result provides a recursive formula to update $P_{k+1}$ given by \eqref{eq:Atheta_TF}. 

\begin{prop}
    \label{prop:Pk_MIL_VDF}
    Let $\lambda \in (0,1]$,
    $\varepsilon > 0$, let $(P_k)_{k=0}^\infty$ be a sequence of $n\times n$ positive-definite matrices, and let $U_k \in \BBR^{n \times n}$ be an orthonormal matrix whose columns are the singular vectors of $P_k.$ 
    Furthermore, let $\psi_k \in \BBR^{p \times n}$ be given by \eqref{eq:psi_k_definition}, let
    $\overline \Lambda_k$ be given by \eqref{eq:LambdaBar_def}, and let
    $\Lambda_k$ be given by \eqref{eq:Lambda_k_def}.
    Then, for all $k\ge0,$ $(P_k)_{k=0}^\infty$ satisfies \eqref{eq:Atheta_TF} if and only if, for all $k\ge0,$
    $(P_k)_{k=0}^\infty$ satisfies
    \begin{align}
        P_{k+1}
            &=
                \overline P_{k} -
                \overline P_{k} \phi_k
                (I_p+\phi_k^\rmT \overline P_{k} \phi_k)^{-1}
                \phi_k^\rmT \overline P_{k},
        \label{eq:P_k_TF_def}
    \end{align}
    where
    \begin{align}
        \overline P_{k}
            &=
                \Lambda_k^{-1} P_{k} \Lambda_k^{-1}.
        \label{eq:Pbar_km1_def}
    \end{align}
\end{prop}

\begin{proof}
    To prove necessity, it follows from \eqref{eq:Atheta_TF} and matrix-inversion lemma, that
    \begin{align}
        P_{k+1}
            &=
                (
                \Lambda_k  P_{k}^{-1} \Lambda_k 
                + 
                \phi_{k}^\rmT \phi_{k}
                )\inv
            \nn \\
            &=
                (\Lambda_k  P_{k}^{-1} \Lambda_k )\inv-
                (\Lambda_k  P_{k}^{-1} \Lambda_k )\inv \phi_k^\rmT
                [I_p + \phi_k (\Lambda_k  P_{k}^{-1} \Lambda_k )\inv \phi_k^\rmT ]\inv
                \phi_k (\Lambda_k  P_{k}^{-1} \Lambda_k)\inv
            \nn \\
            &=
                \overline P_{k} -
                \overline P_{k} \phi_k
                (I_p+\phi_k^\rmT \overline P_{k} \phi_k)^{-1}
                \phi_k^\rmT \overline P_{k},
            \nn
    \end{align}
    where $\overline{P}_k$ is given by \eqref{eq:Pbar_km1_def}. 
    Reversing these steps proves sufficiency.   
\end{proof}

The modified update \eqref{eq:Atheta_TF} is shown to be optimal for a specific cost function in ``A Modified Quadratic Cost Function Supporting Variable-Direction RLS''.

Next, the matrix-forgetting scheme \eqref{eq:Atheta_TF} is shown to prevent the singular values of $P_k$ from diverging.    
Consider the case where, for all $k\ge0$,
\begin{align}
    \psi_k
        =
            \matl{cc}
                \psi_{k,1} &
                0
            \matr,
    \label{eq:psi_k_NP}
\end{align}
where $\psi_{k,1} \in \BBR^{p \times n_1}$, 
that is, the information-rich subspace is spanned by the first $n_1$ columns of $U_k$.
It thus follows from \eqref{eq:Atheta_TF_SVD_flipped} and \eqref{eq:psi_k_NP} that $P_{k+1}^{-1}$ is given by
\begin{align}
    P_{k+1}^{-1}
        =
            U_{k} 
            \matl{cc}
                \lambda \Sigma_{k,1} +\psi_{k,1}^\rmT \psi_{k,1}  & 0 \\
                0               & \Sigma_{k,2}
            \matr
            U_{k}^\rmT.
    \label{eq:Atheta_TF_SVD_Analysis}
\end{align}
It follows from the $(2,2)$ block of \eqref{eq:Atheta_TF_SVD_Analysis} that the last $n-n_1$ information directions and the corresponding singular values are not affected by $\phi_k$.
Furthermore, if $n_1 = n$, that is, new information is present in $\phi_k$  along every information direction, then forgetting is applied to all of the singular values of $P_{k}^{-1}$, and thus variable-direction forgetting specializes to  uniform-direction forgetting, that is, RLS with the update for $P_k$ given by \eqref{eq:Pk_recursive}.

The next result shows that, as in the case of uniform-direction forgetting, $z_k$ converges to zero with variable-direction forgetting for every choice of $\varepsilon>0$, whether or not $\SeqPhi$ is persistently exciting. 

\begin{prop}
    \label{prop:z_converges_VDF}
    For all $k\ge0$, let $\phi_k \in \BBR^{p \times n}$ and $y_k \in \BBR^p$, let $R\in\BBR^{n\times n}$ be positive definite, and let $P_0 = R^{-1}$, $\theta_0 \in \BBR^n$, and  $\lambda \in (0,1]$.
    Furthermore, for all $k\ge0,$ let $P_k$ and $\theta_k$ be given by \eqref{eq:Atheta_TF} and \eqref{eq:thetaUpdate}, respectively.
    Then, 
    \begin{align}
        \lim_{k \to \infty} z_k = 0.
        \label{eq:z_lim_VDF}
    \end{align}
\end{prop}
\begin{proof}  
    Using \eqref{eq:LambdaBar_def}, \eqref{eq:Lambda_k_def}, and $P_k\inv = U_k \Sigma_k U_k^\rmT$, it follows that, for all $k\ge 0$,
    \begin{align}
        \Lambda_k P_k \inv \Lambda_k
            &=
                U_{k} \overline \Lambda_k 
                \Sigma_k 
                \overline \Lambda_k U_{k}  ^\rmT 
            \le
                U_{k} 
                \Sigma_k 
                U_{k}  ^\rmT 
            =
                P_k\inv.
    \end{align}
    For all $k\ge0$, note that $z_k = \phi_k \tilde \theta_k,$ and define $V_k \isdef \tilde \theta_k^\rmT P_k^{-1} \tilde \theta_k$.
    Note that, for all $k\ge 0$ and $\tilde \theta_k \in \BBR^n$, $V_k \ge 0$.
    Furthermore, for all $k\ge0,$
    \begin{align}
        V_{k+1} - V_k
            &=
                \tilde \theta_{k+1}^\rmT P_{k+1}^{-1} \tilde \theta_{k+1} -
                \tilde \theta_k^\rmT P_k^{-1} \tilde \theta_k 
            \nn \\
            &=
                \tilde \theta_k^\rmT \Lambda_k P_k^{-1} \Lambda_k 
                P_{k+1} 
                \Lambda_k P_k^{-1} \Lambda_k \tilde \theta_{k} -
                \tilde \theta_k^\rmT P_k^{-1} \tilde \theta_k 
            \nn \\
            &=
                \tilde \theta_k^\rmT 
                [
                    \Lambda_k P_k^{-1} \Lambda_k 
                    P_{k+1} 
                    \Lambda_k P_k^{-1} \Lambda_k  -
                    P_k^{-1} 
                ]
                \tilde \theta_k 
            \nn \\
            &=
                \tilde \theta_k^\rmT 
                [
                    \Lambda_k P_k^{-1} 
                    (
                        P_{k}  -
                        P_{k} \Lambda_k^{-1} \phi_k
                        (I_p+\phi_k^\rmT \overline P_{k} \phi_k)^{-1}
                        \phi_k^\rmT \Lambda_k^{-1} P_{k} 
                    ) 
                    P_k^{-1} \Lambda_k  -
                    P_k^{-1} 
                ]
                \tilde \theta_k 
            \nn \\
            &=
                \tilde \theta_k^\rmT 
                [
                        \Lambda_k P_k^{-1} \Lambda_k  -
                        \phi_k
                        (I_p+\phi_k^\rmT \overline P_{k} \phi_k)^{-1}
                        \phi_k^\rmT
                    -
                    P_k^{-1} 
                ]
                \tilde \theta_k 
            \nn \\
            &=
                -[
                \tilde \theta_k^\rmT 
                (P_k^{-1} - \Lambda_k P_k^{-1} \Lambda_k )
                \tilde \theta_k +
                z_k
                (I_p+\phi_k^\rmT \overline P_{k} \phi_k)^{-1}
                z_k 
                ]
            \nn \\
            &\le 
                0.
            \nn
    \end{align}
    Note that, since $(V_k)_{k=1}^\infty$ is a nonnegative, nonincreasing sequence, it converges to a nonnegative number.
    Hence, $\lim_{k\to\infty} (V_{k+1} - V_k)=0,$ which implies that $$\lim_{k\to\infty} [ \tilde \theta_k^\rmT 
                (P_k^{-1} - \Lambda_k P_k^{-1} \Lambda_k )
                \tilde \theta_k +
                z_k
                (I_p+\phi_k^\rmT \overline P_{k} \phi_k)^{-1}
                z_k  ] = 0.$$
    Since, for all $k\ge 0$, $P_k^{-1} - \Lambda_k P_k^{-1} \Lambda_k \ge 0$ and
    $(I_p+\phi_k^\rmT \overline P_{k} \phi_k)^{-1}>0$,
    it follows that $\lim_{k \to \infty} z_k = 0.$
\end{proof}

The next result shows that $P_k$ is bounded from above with variable-direction forgetting for every choice of $\varepsilon>0$ in the case where $\SeqPhi$ is persistently exciting.

\begin{prop}
\label{prop:Pkinv_bounds_VDF}
Assume that $(\phi_k)_{k=0}^\infty $ is persistently exciting,
let $N,\alpha,\beta$ be given by Definition \ref{def:persistent_Exc}, let $R\in\BBR^{n\times n}$ be positive definite, define $P_{0} \isdef R^{-1}$, let $\lambda \in (0,1)$, and, for all $k\ge0,$ let $P_k$ be given by \eqref{eq:Atheta_TF}.
Then, for all $k\ge N+1,$ 
    \begin{align}
        \frac{\lambda^N(1-\lambda) \alpha} 
             {1 - \lambda^{N+1} }
             I_n
            \le
                P_k^{-1}.
        \label{eq:Pkinv_bounds_VDF}
    \end{align}
\end{prop}

\begin{proof}
    It follows from \eqref{eq:Atheta_TF}, that, for all $k \ge 0,$
    $
        \Lambda_k P_k \inv \Lambda_k 
            \le
                P_{k+1}\inv
    $
    and 
    $\phi_k^\rmT \phi_k \le P_{k+1} \inv.$
    Next, using \eqref{eq:Lambda_k_def} and $P_k\inv = U_k \Sigma_k U_k^\rmT$, it follows that, for all $k \ge 0$,
    \begin{align}
        \lambda P_k\inv
            &=   
                \lambda U_{k} 
                \Sigma_k 
                U_{k}  ^\rmT 
            \le
                U_{k} \overline \Lambda_k 
                \Sigma_k 
                \overline \Lambda_k U_{k}  ^\rmT 
            =
                \Lambda_k P_k \inv \Lambda_k
            \le 
                P_{k+1}\inv. \nn
    \end{align}
    Finally, for all $k \ge N+1$,
    \begin{align}
        \alpha I_n
                &\le
                    \sum_{i=k-N-1}^{k-1}
                        \phi_i^\rmT  \phi_i\nn\\
            &\le 
                \sum_{i=k-N}^{k} P_{i}^{-1}\nn\\
                &\le
                ( \lambda^{-N} + \cdots + 1 ) P_{k}^{-1}\nn\\
            &=
                \dfrac{1-\lambda^{N+1}}{\lambda^N (1-\lambda)} P_{k}^{-1}, \nn 
    \end{align}
    which proves \eqref{eq:Pkinv_bounds_VDF}.
    %
\end{proof}

The next two examples consider variable-direction forgetting in the case where $\SeqPhi$ is not persistently exciting.
In these examples, $P_k$ is bounded, 
$z_k$ converges to zero, and $\theta_k$ converges, although not to the true value $\theta.$

\begin{exam}
    \label{exam:Persistency_InformationContent_with_TF}
    \textit{Variable-direction forgetting for a regressor lacking persistent excitation.}
    Reconsider Example \ref{exam:Persistency_InformationContent}.
    Let $P_0 = I_{10}$, and $P_k^{-1}$ be given by \eqref{eq:P_k_TF_def}, where $\varepsilon = 10^{-8}$. 
    Figure \ref{fig:CSM_forgetting_Persistency_InformationContent_Targeted} shows the information content $| {\rm col}_i (\psi_{k}) |$ and the singular values of the $P_k^{-1}$ for several values of $\lambda$.
    Note that the information-rich subspace is six dimensional due to the presence of three harmonics in $u_k$ as shown by six relatively large components of $\psi_k$ and the singular values that correspond to the singular vectors not in the information-rich subspace do not converge to zero. 
    \EndExample
\end{exam}

\begin{figure}[h!]
    \centering
    \includegraphics[width=1 \textwidth]{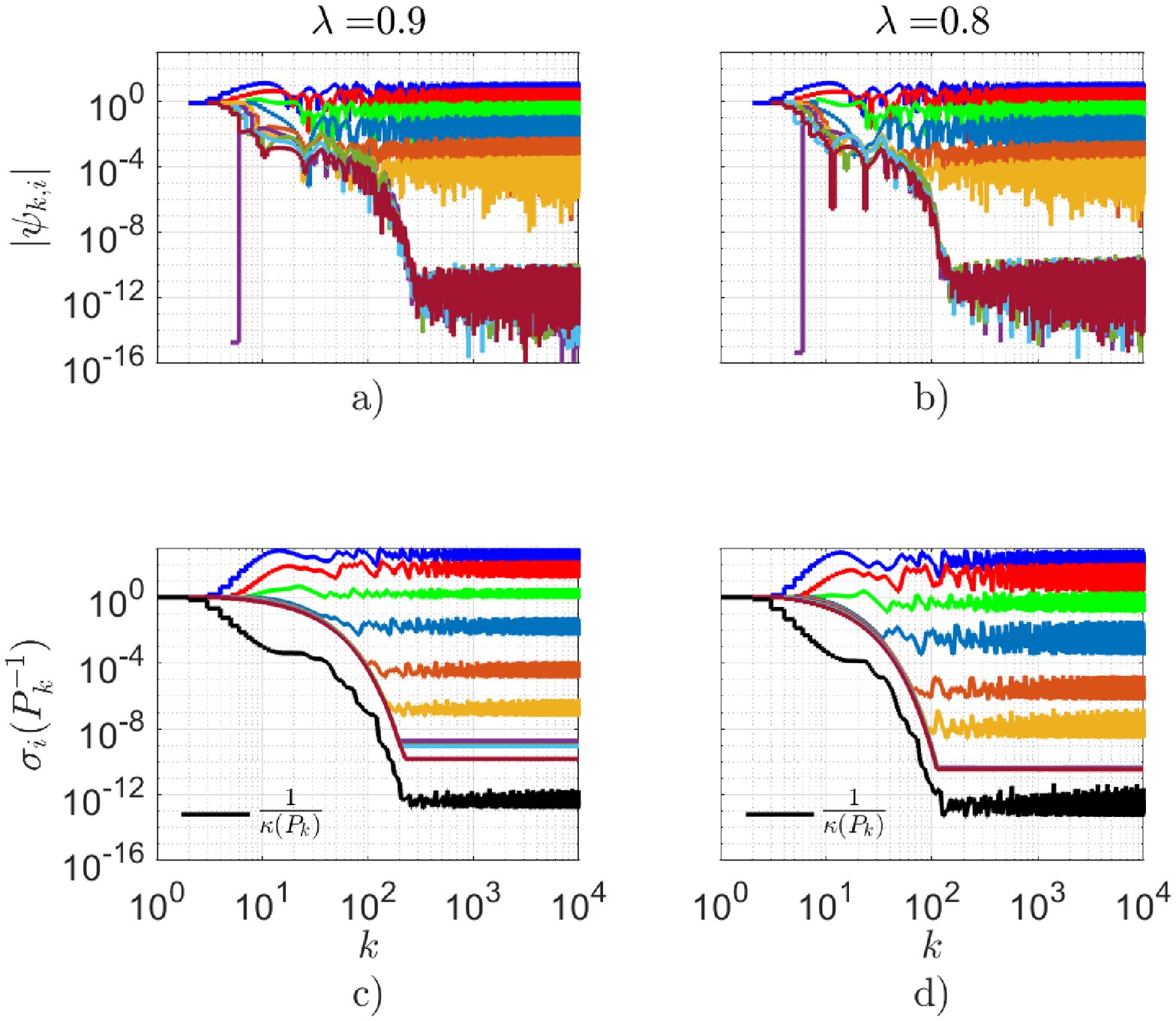}
    \caption{
    Example \ref{exam:Persistency_InformationContent_with_TF}. 
    Variable-direction forgetting for a regressor lacking persistent excitation.
    a) and b) show the information content $\| \psi_k \|$ for $\lambda = 0.9$ and $\lambda = 0.8$.
    c) and d) show the singular values of $P_k^{-1}$ for $\lambda = 0.9$ and $\lambda = 0.8$.
    The inverse of the condition number of $P_k$ is shown in black. 
    Note that, for $\lambda<1$, the singular values that correspond to the singular vectors not in the information-rich subspace do not converge to zero.
    }
    \label{fig:CSM_forgetting_Persistency_InformationContent_Targeted}
\end{figure}

\begin{exam}
    \label{exam:TF_effect_on_theta_k}
    \textit{Effect of variable-direction forgetting on $\theta_k$.}
    Reconsider Example \ref{exam:LackofPE_theta_k}.
    Let $P_0 = I_{10}$, and $P_k^{-1}$ be given by \eqref{eq:P_k_TF_def}, where $\varepsilon = 10^{-8}$. 
    Figure \ref{fig:CSM_forgetting_5thOrder_fit_Targeted} shows the predicted error $z_k $, the norm of the parameter error $\tilde \theta_k$, and 
    the singular values and the condition number of $P_k$.
    Note that the $\tilde \theta_k$ does not converge to zero and, unlike uniform-direction forgetting, all of the singular values of $P_k$ remain bounded and $\theta_k$ is bounded.  
    \EndExample
\end{exam}

\begin{figure}[h!]
    \centering
    \includegraphics[width=1 \textwidth]{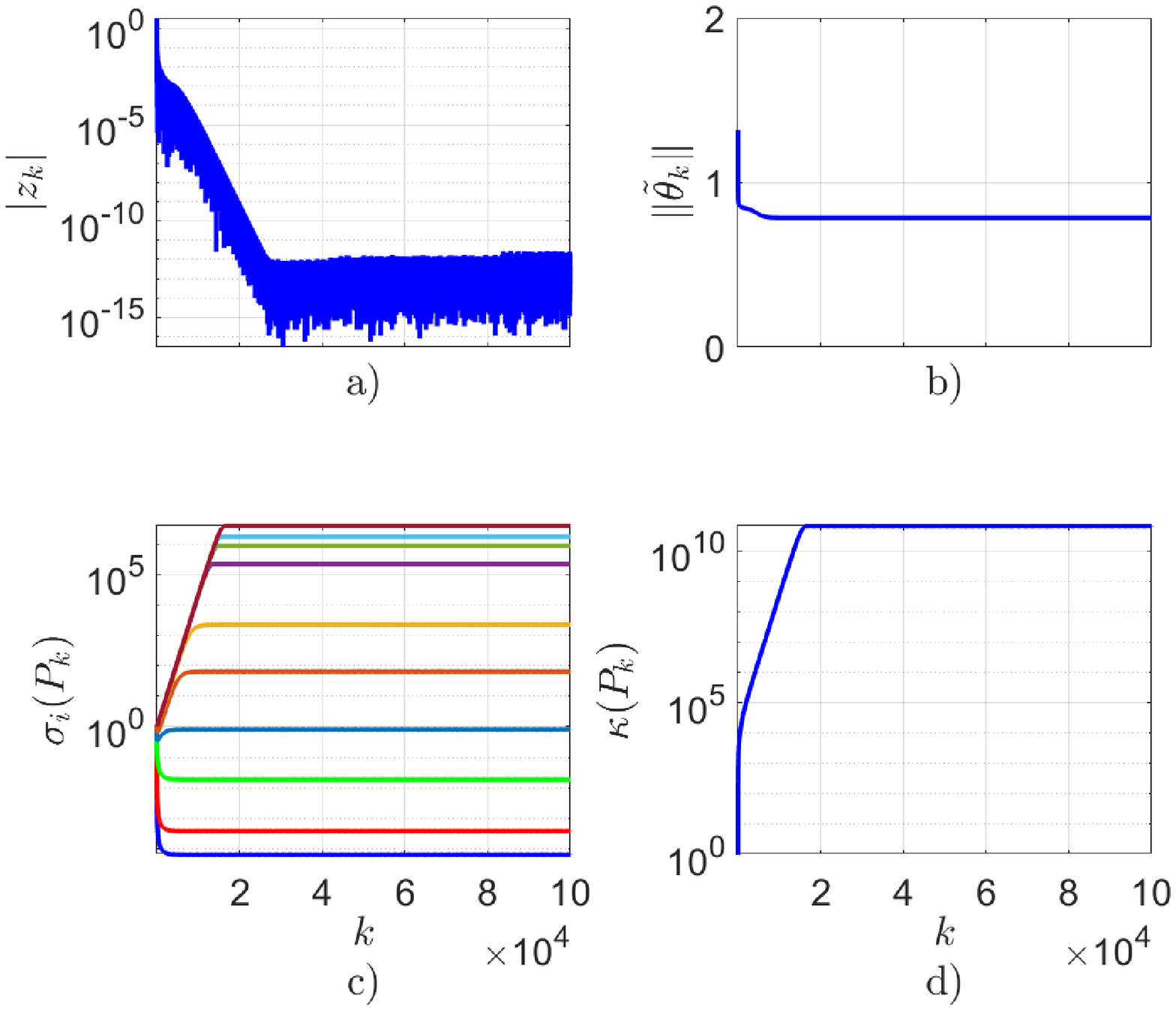}
    \caption{
    Example \ref{exam:TF_effect_on_theta_k}.
    Effect of variable-direction forgetting on $\theta_k$. 
    a) shows the predicted error $z_k $,  
    b) shows the norm of the parameter error $\tilde \theta_k$, 
    c) shows the singular values of $P_k$, and
    d) shows the condition number of $P_k$.
    Note that all of the singular values of $P_k$ remain bounded.
    }
    \label{fig:CSM_forgetting_5thOrder_fit_Targeted}
\end{figure}


\section{Concluding Remarks}

This tutorial article presented a self-contained exposition of uniform-direction  and variable-direction forgetting within the context of RLS.
It was shown that, in the case of persistent excitation without forgetting, the parameter estimates converge asymptotically, whereas, with forgetting, the parameter estimates converge geometrically.
Numerical examples were presented to illustrate this behavior.

In the case where forgetting is used but the excitation is not persistent, it was shown that forgetting is enforced in all information directions, whether or not new information is present along these directions. 
Consequently, the parameter estimates converge, but not necessarily to their true values; furthermore, the matrix $P_k$ diverges, leading to numerical instability.
This phenomenon was traced to the divergence of the singular values of $P_k$ corresponding to singular vectors that are orthogonal to the information-rich subspace.

In order to address this problem, a data-dependent forgetting matrix was constructed to restrict forgetting to the information-rich subspace. 
The RLS cost function that corresponds to this extension of RLS was presented.
Numerical examples showed that this variable-direction forgetting technique prevents $P_k$ from diverging under lack of persistent excitation.

Since RLS is fundamentally least squares optimization, its estimates are not consistent in the case of sensor noise \cite{eykhoff1974system}.
An open problem is thus to develop extensions of RLS that provide consistent parameter estimates in the presence of errors-in-variable noise arising in system identification problems \cite{soderstrom2018errors}.

\section*{Acknowledgments}

This research was partially supported by AFOSR under DDDAS grant FA9550-16-1-0071
(Dynamic Data-Driven Applications Systems \href{http://www.1dddas.org/}{http://www.1dddas.org/}).


\bibliographystyle{IEEEtran.bst}
\bibliography{recursive,references1}

\clearpage 
\section*{Sidebar: Summary}
Learning depends on the ability to acquire and assimilate new information.  This ability depends---somewhat counterintuitively---on the ability to forget.
In particular, effective forgetting requires the ability to recognize and utilize new information to order to update a system model.
This article is a tutorial on forgetting within the context of recursive least squares (RLS). 
To do this, RLS is first presented in its classical form, which employs  uniform-direction forgetting.
Next, examples are given to motivate the need for variable-direction forgetting, especially in cases where the excitation is not persistent. 
Some of these results are well known, whereas others complement the prior literature.
The goal is to provide a self-contained tutorial of the main ideas and techniques for students and researchers whose research may benefit from variable-direction forgetting.

\clearpage 
\setcounter{equation}{0}
\renewcommand{\theequation}{S\arabic{equation}}
\setcounter{lemma}{0}
\renewcommand{\thelem}{S\arabic{lem}}
\section*{Sidebar:  Three Useful Lemmas}
\begin{lem}
    \label{lemma:y_in_X}
    Let $X \in \BBR^{n \times p}$ and $y \in \BBR^{ n}$, and let $W \in \BBR^{p \times p}$ be positive definite.
    Then,
    \begin{align}
        (I_n + X W X^\rmT)^{-1} y
        \in
        \SR([X \ \ y]).
        \label{eq:y_in_X}
    \end{align}
\end{lem}

\begin{proof}
    Note that
    \begin{align}
        y
            &\in
                \SR([X \ \ y]) 
            \nn \\
            &=
                \SR [X  \ \ y + X W X^\rmT y]
            \nn \\
            &=
                \SR
                \left(
                [X  \ \ (I_n + X W X^\rmT) y]
                \matl{cc}
                    I_p + W X^\rmT X & 0\\
                    0 & 1
                \matr
                \right) 
            \nn \\
            &=
                \SR([X(I_p +  W X^\rmT X) \ \ (I_n + X W X^\rmT) y])
            \nn \\
            &=
                \SR([(I_n + X W X^\rmT) X \ \ (I_n + X W X^\rmT) y])
            \nn \\
            &=
                (I_n + X W X^\rmT) \SR([ X \ \ y]), \nn
    \end{align}
    which implies \eqref{eq:y_in_X}.
    %
    %
\end{proof}

\begin{lem}
\label{prop:IminusAlambda}
Let $A \in \BBR^{n \times n}$ be positive semidefinite, and let $\lambda >0$.
Then,
\begin{align}
    I_n - A (\lambda I_n + A)^{-1} > 0.
    \label{eq:IminusAlambda}
\end{align}
\end{lem}
\begin{proof}
    Write $A = S D S^\rmT$, where $D= {\rm diag}(d_1,\ldots,d_n)$ is diagonal and $S$ is unitary.
    For all $i \in \{1, \ldots, n \},$ $d_i\ge0,$ and thus $\tfrac{d_i}{\lambda+d_i} < 1$. 
    Hence,
    \begin{align}
        D(\lambda I_n + D)^{-1} 
            =  
                {\rm diag} \left( \tfrac{d_1}{\lambda+d_1}, \ldots , \tfrac{d_n}{\lambda+d_n} \right) 
            < I_n.
        \label{eq:D_ineq}
    \end{align}
    Pre-multiplying and post-multiplying \eqref{eq:D_ineq} by $S$ and $S^\rmT$, respectively, yields \eqref{eq:IminusAlambda}.
\end{proof}

\begin{lem}
\label{Lem:Lemma_A_identity}
Let $A \in \BBR^{n \times n}$ be positive semidefinite, and let $\lambda >0$.
Then,
\begin{align}
    I_n - 
    \dfrac{1}{\lambda}
    \left( A - A (\lambda I_n + A)^{-1} A \right) > 0.
    \label{eq:Lemma_A_identity}
\end{align}
\end{lem}

\begin{proof}
    Write $A = S D S^\rmT$, where $D= {\rm diag}(d_1,\ldots,d_n)$ is diagonal and $S$ is unitary.
    For all $i \in \{1, \ldots, n \},$ $d_i\ge0,$ and thus $\tfrac{d_i}{\lambda+d_i} < 1$. 
    Hence,
    \begin{align}
        \dfrac{1}{\lambda}
        \left( 
            D - D(\lambda I_n + D)^{-1} D 
        \right)
            =  
                {\rm diag} \left( \tfrac{d_1}{\lambda+d_1}, \ldots , \tfrac{d_n}{\lambda+d_n} \right) 
            < I_n.
        \label{eq:D_ineq2}
    \end{align}
    Pre-multiplying and post-multiplying \eqref{eq:D_ineq2} by $S$ and $S^\rmT$, respectively, yields \eqref{eq:Lemma_A_identity}.
\end{proof}



\clearpage 
\setcounter{equation}{0}
\renewcommand{\theequation}{S\arabic{equation}}
\setcounter{prop}{0}
\renewcommand{\theprop}{S\arabic{prop}}
\section*{Sidebar:  RLS as a One-Step Optimal Predictor}
Consider the linear system
\begin{align}
    x_{k+1}
        &= 
            A_k x_{k} + B_k u_{k} + w_{1,k}, \\
    y_{k} 
        &=
            C_k x_{k} + w_{2,k},
\end{align}
where, for all $k\ge0,$
$x_k \in \BBR^n$, 
$u_k \in \BBR^m$,
$y_k \in \BBR^p$, and
$A_k,B_k,C_k$ are real matrices of appropriate sizes. 
The input $u_k$ and output $y_k$ are assumed to be measured.
The process noise $w_{1,k} \in \BBR^{n}$ and the sensor noise $w_{2,k} \in \BBR^p$ are zero-mean white noise processes with variances
$\BBE[w_{1,k} w_{1,k}^\rmT ] = Q_k$  and
$\BBE[w_{2,k} w_{2,k}^\rmT ] = R_k$, respectively.
%
%
The expected value of the initial state is assumed to be $\overline x_0$, and the variance of the initial state is $P_0$, that is, 
$\BBE[x_{0}] = \overline x_0$ and 
$\BBE[(x_{0}-\overline x_0) (x_{0}-\overline x_0)^\rmT ] = P_0$. 
The objective is to estimate the state $x_k$ given the measurements of $u_k$ and $y_k$.

To estimate $x_k$, consider the estimator
\begin{align}
    \hat x_{k+1}
        &= 
            A_k \hat x_{k} + B_k  u_{k} + 
            K_{k} ( y_{k} - C_k \hat x_{k} ), 
\end{align}
where 
$\hat x_k$ is the estimate of $x_k$ at step $k$ and $\hat x_0 = \overline x_0$.
The matrix $K_k$ is constructed as follows. 
%
%
%
Define the \textit{state-estimate error} $e_{k} \isdef x_{k} - \hat x_{k}$ and the \textit{state error covariance} $P_{k} \isdef \BBE[ e_{k} e_{k}^\rmT] \in \BBR^{n \times n}$.
Then, $e_{k}$ and $P_{k}$ satisfy
\begin{align}
    e_{k+1} 
        &=
            (A_k -  K_{k}C_k) e_{k} + w_{1,k} -  K_{k} w_{2,k}, \\
    P_{k+1} 
        &=
            A_k P_{k} A_k^\rmT + Q_k + 
            K_{k} 
            \left(R_k + C_k P_{k} C_k^\rmT \right)
            K_{k}^\rmT  -
            A_k P_{k} C_k^\rmT K_{k}^\rmT - 
            C_k P_{k} A_k^\rmT .
    \label{eq:P_kp1_1SP}
\end{align}
\begin{prop}
    Let $P_{k+1}$ be given by \eqref{eq:P_kp1_1SP}.
    The matrix $K_k$ that minimizes ${\rm tr \ }P_{k+1}$ is given by
    \begin{align}
        K_k
            &=
                A_k P _{k} C_k^\rmT 
                \left(R_k + C_k P_{k} C_k^\rmT \right)^{-1},
        \label{eq:OSP_Kk}
    \end{align}
    and the minimized state-error covariance $P_k$ is updated as
    \begin{align}
        P_{k+1}
            &=
                A_k P_k A_k^\rmT + Q_k -
                A_k P_k C_k^\rmT
                \left(R_k + C_k P_{k} C_k^\rmT \right)^{-1}
                C_k P_k A_k^\rmT. 
        \label{eq:1SP_Pk_update}
    \end{align}
\end{prop}
\begin{proof}
    See \cite{AseemPredictorCSM_SB}.
\end{proof}

Let $A_k = I_n$, $B_k = 0$, $C_k = \phi_k$, $Q_k = 0,$ and $R_k =  I_p$. Then, 
\begin{align}
    \hat x_{k+1}
        &= 
            \hat x_{k} + 
            P _{k} \phi_k^\rmT 
            \left(  I_p + \phi_k P _{k} \phi_k^\rmT  \right)^{-1}
            ( y_{k} - \phi_k \hat x_{k} ), 
    \label{eq:1SP_x_RLS}
    \\
    P_{k+1}
            &=
                P_k -
                P_k \phi_k^\rmT
                \left( I_p + \phi_k P _{k} \phi_k^\rmT  \right)^{-1}
                \phi_k P_k .
    \label{eq:1SP_P_RLS}
\end{align}
Note that \eqref{eq:theta_update_WithInverse}, \eqref{eq:P_update_WithInverse} with $\lambda=1$ have the same form as \eqref{eq:1SP_x_RLS}, \eqref{eq:1SP_P_RLS}. 
In particular, RLS without forgetting is the state estimator for the linear time-varying system with $A_k = I_n$, $B_k = 0$, $C_k = \phi_k$, $Q_k = 0,$ and $R_k =  I_p$.

\clearpage 
\setcounter{equation}{0}
\renewcommand{\theequation}{S\arabic{equation}}
\setcounter{prop}{0}
\renewcommand{\theprop}{S\arabic{prop}}

\section*{Sidebar:  RLS as a Maximum Likelihood Estimator }

Let $k\ge0$ and, for all $i\in\{0,1,\ldots,k\},$ consider the process
\begin{align}
    y_i
        =
            \phi_i \theta_{\rm true} + v_i,
    \label{eq:process_ML}
\end{align}
where
$\theta_{\rm true} \in \BBR^n$ is the unknown parameter, 
$\phi_i \in \BBR^{p \times n}$ is the regressor matrix, 
$v_i \in \BBR^p$ is the measurement noise, and
$y_i \in \BBR^p$ is the measurement. 
The goal is to estimate $\theta_{\rm true}$ using the data  $(\phi_i)_{i=0}^k$ and $(y_i)_{i=0}^k$.

Let $\theta_{\rm true}$ be modeled by the $n$-dimensional, real-valued normal random variable $\Theta$ with mean $\theta_0 \in \BBR^n$ and covariance $(\lambda ^{k+1}R)\inv$, where $\lambda \in (0,1]$ and $R \in \BBR^{n \times n}$ is positive definite. 
For $\theta\in\BBR^n,$ the density of $\Theta$ is thus given by
\begin{align}
    f_{\Theta}(\theta)
        = 
            \dfrac{1}
            {\sqrt{(2 \pi)^n {\rm det}\, {(\lambda ^{k+1}R)\inv}}}
            {\rm exp} 
            [
                -\half
                (\theta-\theta_0)^\rmT 
                \lambda^{k+1} R
                (\theta-\theta_0)
            ].
    \label{eq:theta_pdf}
\end{align}
%
%
%
For all $i \in \{0,1, \ldots, k \}$, assume that $v_i$ is a sample of the zero-mean, $p$-dimensional, real-valued normal random variable $V_i$ with covariance $\lambda^{i-k} I_p$.
For $v_i\in\BBR^p,$ the density of $V_i$ is thus given by
\begin{align}
    f_{V_i}(v_i)
        = 
            \dfrac{1}{\sqrt{(2 \pi)^p \lambda^{i-k}}}
            {\rm exp} 
            (
                -\half
                v_i^\rmT 
                \lambda^{k-i} I_p
                v_i
            ).
    \label{eq:nu_i_pdf}
\end{align}
Assume that $V_0,V_1,\ldots,V_k$ are independent.




Since $\theta_{\rm true}$ and $v_i$ are modeled as normal random variables, it follows from \eqref{eq:process_ML} that 
$y_i$ is a sample of the $p$-dimensional, real-valued normal random variable $Y_i = \phi_i\theta_{\rm true}+V_i$.
%
Note that, since $V_0, V_1, \ldots, V_k$ are independent, it follows that $Y_0, Y_1, \ldots, Y_k$ are independent. 
%
%
Using \eqref{eq:process_ML} and \eqref{eq:nu_i_pdf}, it thus follows that 
\begin{align}
    f_{Y_i \vert \theta }(y_i)
        = 
            \dfrac{1}{\sqrt{(2 \pi)^p \lambda^{i-k}}}
            {\rm exp} 
            [
                -\half
                (y_i-\phi_k \theta)^\rmT 
                \lambda^{k-i} I_p
                (y_i-\phi_k \theta) 
            ],
    \label{eq:Yi_given_theta_pdf}
\end{align}
where $f_{Y_i \vert \theta }(y_i)$ is the density of the random variable $Y_i$ conditions on $\Theta$ taking the value $\theta$.



It follows from Bayes' rule \cite[p. 413]{Betrseka_prob_book} that
\begin{align}
    f_{\Theta \vert \{ y_0, \ldots, y_k \} }(\theta)
        &=
            \alpha \inv 
            f_{\Theta}(\theta) 
            \prod_{i=0}^kf_{Y_i \vert \theta} (y_i),
    \label{eq:theta_given_y_seq_pdf}
\end{align}
where 
\begin{align}
    \alpha 
        \isdef
              \int_{\BBR^n }f_{\Theta}(\theta) 
                    \prod_{i=0}^k f_{Y_i \vert \theta} (y_i)\, 
                    \rmd \theta.
\end{align}
%
%
Substituting \eqref{eq:theta_pdf} and \eqref{eq:Yi_given_theta_pdf} into \eqref{eq:theta_given_y_seq_pdf}, it follows that 
\begin{align}
    f_{\Theta \vert \{ y_0, \ldots, y_k \} }(\theta)
        =
            \beta
            {\rm exp}
            \left[
                \sum_{i=0}^k
                    -\half
                    \lambda^{k-i}
                    (y_i-\phi_k \theta)^\rmT 
                    (y_i-\phi_k \theta) 
                    -\half
                    \lambda^{k+1}
                    (\theta-\theta_0)^\rmT 
                    R
                    (\theta-\theta_0)
            \right],
    \label{eq:J_MLE}
\end{align}
where
\begin{align}
    \beta 
        \isdef
            \dfrac{1}
            {\alpha\sqrt{(2 \pi)^p \lambda^{i-k} }}
            \dfrac{1}
            {\sqrt{(2 \pi)^n {\rm det}\, {(\lambda ^{k+1}R)\inv}}}.
\end{align}



Finally, the \textit{maximum likelihood estimate} of $\theta_{\rm true}$ is given by the maximizer of \eqref{eq:J_MLE}, that is,
\begin{align}
    \theta_{{\rm ML}}
        =
            \underset{ \theta \in \BBR^n  }{\operatorname{argmax}} \
            f_{\Theta \vert \{ y_0, \ldots, y_k \} }(\theta).
    \label{eq:theta_MLK}
\end{align}
In fact, $\theta_{\rm ML} = \underset{ \theta \in \BBR^n  }{\operatorname{argmin}} \
            J_{k}(\theta)$, 
where $J_k(\theta)$ is given by \eqref{eq:J_LS}.
Therefore, RLS with forgetting can be interpreted as the maximum likelihood estimator of the random variable $\Theta.$



\clearpage 
\setcounter{equation}{0}
\renewcommand{\theequation}{S\arabic{equation}}
\setcounter{prop}{0}
\renewcommand{\theprop}{S\arabic{prop}}

\section{Sidebar:  Toward Matrix Forgetting}

In \cite{Kreisselmeier_SB},  $P_k\inv$ is updated by
%
\begin{align}
    P_{k+1}^{-1} 
        &=
            (I_n + M_k P_k)  P_k^{-1} +
            \phi_{k}^\rmT \phi_{k},
    \label{eq:Pk_Update_Kreiss}
\end{align}
where $M_k \in \BBR^{n \times n}$ is chosen to guarantee asymptotic stability and boundedness. 
Two choices of matrix $M_k$ are considered.
In the first case, 
\begin{align}
    M_k 
        &\isdef
            -(1-\lambda) (I-\alpha P_k)^N P_k^{-1},
    \label{eq:Kreiss_Stab1}
\end{align}
where
$\lambda \in (0,1)$, $\alpha>0$, and $N$ is an odd, positive integer. 
In the second case, 
\begin{align}
    M_k 
        &=
            -(1-\lambda) 
            (P_k^{-1}-\alpha I_n)^N 
            (P_k^{-1}+\beta I_n)^{-N}
            P_k^{-1},
    \label{eq:Kreiss_Stab2}
\end{align}
where $\lambda \in (0,1)$, $\alpha>0$, $\beta \ge 0$, and $N$ is an odd, positive integer. 
Note that RLS with constant forgetting is obtained by setting $M_k = (\lambda-1) P_k^{-1} $ in \eqref{eq:Pk_Update_Kreiss}.

\begin{prop}
    \label{prop:Kreiss_paper}
    Consider \eqref{eq:Pk_Update_Kreiss} with \eqref{eq:Kreiss_Stab1} or \eqref{eq:Kreiss_Stab2}.
    Let $P_0$ be symmetric and nonsingular.
    Then, the following statements hold:
    \begin{enumerate}
        \item 
            For all $k \ge 0$, $P_k$ is symmetric and nonsingular.
        \item
            If $P_0^{-1} \ge \tfrac{\alpha}{2}I_n$, then, 
            $P_k^{-1} = \alpha I$ is an asymptotically stable equilibrium of \eqref{eq:Pk_Update_Kreiss}.
        \item
            If $P_0^{-1} \ge \alpha I_n$, then, for all $k \ge 0$,
            $P_k^{-1} \ge \alpha I_n$. 
        \item
            If $P_0^{-1} \ge \alpha I_n$ and, for all $k \ge 0$, $\phi_k$ is bounded, then 
            $P_k^{-1}$ is bounded.
        \item
            If $P_0^{-1} \ge \alpha I_n$ and $\phi_k$ is persistently exciting, then there exists $k_0>0$ such that, for all $k\ge k_0,$
            $P_k^{-1} > \alpha I_n$.
    \end{enumerate}
\end{prop}
\begin{proof}
    See \cite{kreisselmeier1990stabilized}.
\end{proof}

The main goal of \eqref{eq:Pk_Update_Kreiss} is stabilization of $P_k$ in the case where $\SeqPhi$ is not persistently exciting.
Proposition \ref{prop:Kreiss_paper} implies that $P_k$ remains bounded whether or not $\SeqPhi$ is persistent. 
However, \eqref{eq:Pk_Update_Kreiss} is not designed to implement forgetting. 
Furthermore, note that \eqref{eq:Pk_Update_Kreiss} requires the computation of the inverse of an $n \times n$ matrix at each step.

An alternative directional forgetting scheme given in \cite{Cao_SB} considers the update  
\begin{align}
    P_{k+1}^{-1} 
        &=
            M_k  P_k^{-1} +
            \phi_{k}^\rmT \phi_{k},
    \label{eq:Pk_update_Cao}
\end{align}
where $M_k \in \BBR^{n \times n}$ is designed to apply forgetting to a specific subspace.
In the case of a scalar measurement, that is, $p=1$, $P_k^{-1}$ is decomposed as
\begin{align}
    P_k^{-1}
        = 
            P_{1,k}^{-1} +
            P_{2,k}^{-1},
\end{align}
where $P_{1,k}^{-1}$ is chosen such that $P_{1,k}^{-1} \phi_k^\rmT = 0$, that is, $\phi_k^\rmT$ is in the null space of $P_{1,k}^{-1}$.
Next, forgetting is restricted to $P_{2,k}^{-1}$, that is, 
\begin{align}
    P_{k+1}^{-1} 
        &=
            P_{1,k}^{-1} +
            \lambda P_{2,k}^{-1} +
            \phi_{k}^\rmT \phi_{k}.
    \label{eq:Pk_update_Cao_RF}
\end{align}
The matrix $P_{2,k}^{-1}$ is chosen to be positive semidefinite with rank $1$ by using
\begin{align}
    P_{2,k}^{-1}
        \isdef
            P_k^{-1} \phi_k^{\rmT}
            \left( \phi_k P_{k}^{-1} \phi_k^\rmT \right)^{-1}
            \phi_k P_k^{-1},
    \label{eq:Pk_update_Cao_P2}
\end{align}
and thus $ P_{1,k}^{-1} = P_k^{-1} - P_{2,k}^{-1}.$
%
Finally, it follows from \eqref{eq:Pk_update_Cao}, \eqref{eq:Pk_update_Cao_RF}, and \eqref{eq:Pk_update_Cao_P2} that
\begin{align}
    M_k
        =
            I_n - 
            (1-\lambda)
            \left( \phi_k P_{k}^{-1} \phi_k^\rmT \right)^{-1}
            P_k^{-1} \phi_{k}^\rmT \phi_{k}
\end{align}
and $P_{k+1}$ is computed as
\begin{align}
    \overline{P}_k
        &=
            \begin{cases}
                P_k +
                \dfrac{1-\lambda}{\lambda}
                \left( \phi_k P_{k}^{-1} \phi_k^\rmT \right)^{-1}
                \phi_{k}^\rmT \phi_{k}
                ,&
                \phi_k \neq 0, \\
                P_k 
                ,&
                \phi_k = 0, \\
            \end{cases}
        \\
    P_{k+1}
        &=
            \overline P_{k} -
            \overline P_{k} \phi_k
            (1+\phi_k^\rmT \overline P_{k} \phi_k)^{-1}
            \phi_k^\rmT \overline P_{k}.
\end{align}

It is shown in \cite{Cao_SB} that, if $P_k \inv$ is positive definite, then, for all $\lambda \in (0,1]$, $M_k P_k \inv$ is positive definite.
Furthermore, if, for all $k\ge 0$, $\phi_k$ is bounded, then there exists $\beta > 0$ such that, for all $k\ge0$, $P_k < \beta I_n $.



\clearpage 
\setcounter{equation}{0}
\renewcommand{\theequation}{S\arabic{equation}}
\setcounter{theo}{0}
\renewcommand{\thetheo}{S\arabic{theo}}

\section*{Sidebar: A  Cost Function for Variable-Direction RLS}
\begin{theo}\label{theorem_RLS_VDF}
    For all $k\ge0$, let $\phi_k \in \BBR^{p \times n}$ and $y_k \in \BBR^p$.
    Furthermore, let $R\in\BBR^{n\times n}$ be positive definite, let $\lambda \in (0,1],$ and, for all $k\ge0,$  let $P_k$ be given by 
    \begin{align}
        P_{k+1}^{-1}
            =
                \Lambda_k  P_{k}^{-1} \Lambda_k 
                +
                \phi_k ^{\rmT}\phi_k,
        \label{eq:SB_Pkinv_recursive_VDF}
    \end{align}
    where $P_0 \isdef R\inv$ and let $\Lambda_k$ be given by \eqref{eq:Lambda_k_def}.
    In addition, let $\theta_0 \in \BBR^n$, and define
    \begin{align} 
        J_k({\hat\theta})
            &\isdef
                \sum_{i=0}^k
                (y_{i} - \phi_{i} {\hat\theta})^\rmT
                (y_{i} - \phi_{i} {\hat\theta}) 
                + 
                ({\hat\theta}-\theta_0) ^\rmT
                R_k
                ({\hat\theta}-\theta_0),
        \label{eq:J_LS_VDF}
    \end{align}
    where, for all $k \ge 0$,
    \begin{align}
        R_k
            =
                R_{k-1} + \Lambda_k P_{k}^{-1} \Lambda_k - P_{k}^{-1},
        \label{eq:Rk_recursive}
    \end{align}
    where $R_{-1} \isdef R$.
    Then, for all $k \ge 0$, \eqref{eq:J_LS_VDF} has a unique global minimizer
    \begin{align}
        \theta_{k+1}
            =
                \underset{ \hat\theta \in \BBR^n  }{\operatorname{argmin}} \
                J_k({\hat\theta}),
       \label{eq:theta_minimizer_def_VDF}
    \end{align}
    which is given by 
    \begin{align}
        \theta_{k+1}
            =
                \theta_k +
                P_{k+1} \phi_k^\rmT (y_k -  \phi_k \theta_k) +
                P_{k+1} (R_k - R_{k-1}) (\theta_0-\theta_k).
        \label{eq:theta_update_VDF}
    \end{align}
\end{theo}

\begin{proof}
    Note that, for all $k \ge 0$,
    \begin{align}
        J_k({\hat \theta} )
            =
                \hat \theta ^\rmT A_k \hat \theta + 
                \hat \theta ^\rmT b_k  + 
                c_k, \nn
    \end{align}
    where 
    \begin{align}
        A_k
            &\isdef
                \sum_{i=0}^k  \phi_i^\rmT \phi_i + R_k, 
            \label{eq:A_k_VDF}
            \\
        b_k
            &\isdef
                \sum_{i=0}^k - \phi_i^\rmT y_i -
                R_k \theta_0,
            \label{eq:b_k_VDF}
            \\
        c_k
            &\isdef
                \sum_{i=0}^k 
                y_i^\rmT y_i +
                \theta_0^\rmT R_k \theta_0.\nn
    \end{align}
    Using \eqref{eq:Rk_recursive}, \eqref{eq:A_k_VDF}, and \eqref{eq:b_k_VDF}, it follows that, for all $k \ge 0,$
    \begin{align}
        A_k 
            &=
                A_{k-1} + \Lambda_k P_{k}^{-1} \Lambda_k - P_{k}^{-1} + \phi_k^\rmT \phi_k, 
        \label{eq:A_k_recursive_VDF}
        \\
        b_k
            &=
                b_{k-1} - \phi_k^\rmT y_k - (R_k - R_{k-1}) \theta_0,
        \label{eq:b_k_recursive_VDF}
    \end{align}
    where $A_{-1} \isdef R$ and $b_{-1} \isdef -R \theta_0$.
    Using \eqref{eq:SB_Pkinv_recursive_VDF} and \eqref{eq:A_k_recursive_VDF}, it follows that, for all $k \ge 0$,
    \begin{align}
        A_k - P_{k+1}^{-1}
            &=
                A_{k-1}  - P_{k}\inv
            \nn \\
            &=
                A_{-1} - P_0\inv
            \nn \\
            &=
                0. \nn
    \end{align}
    It follows from \eqref{eq:Atheta_TF_SVD} that, for all $k \ge 0$, $P_{k+1}\inv$ is positive definite, and thus $A_k$ is positive definite. 
    Furthermore, for all $k \ge 0,$ $A_k$ is given by %
    \begin{align}
        A_k 
            =
                \Lambda_k A_{k-1} \Lambda_k + \phi_k^\rmT \phi_k. \nn
    \end{align}
    
    Finally, since $A_k$ is positive definite, it follows from Lemma 1 in \cite{Aseem_RLS_CSM_SB} that 
    \begin{align}
        \theta_{k+1}
            &=
                -A_k \inv b_k 
            \nn \\
            &=
                -A_k \inv 
                (b_{k-1} - \phi_k^\rmT y_k - (R_k - R_{k-1}) \theta_0)
            \nn \\
            &=
                -A_k \inv 
                (-A_{k-1} \theta_k - \phi_k^\rmT y_k - (R_k - R_{k-1}) \theta_0)
            \nn \\
            &=
                A_k \inv 
                ((A_{k} - R_k + R_{k-1} - \phi_k^\rmT \phi_k) \theta_k +
                \phi_k^\rmT y_k + 
                (R_k - R_{k-1}) \theta_0)
            \nn \\
            &=
                A_k \inv 
                (A_{k} \theta_k +
                \phi_k^\rmT (y_k -  \phi_k \theta_k) +
                (R_k - R_{k-1}) (\theta_0-\theta_k)
            \nn \\
            &=
                \theta_k +
                A_k \inv \phi_k^\rmT (y_k -  \phi_k \theta_k) +
                A_k \inv (R_k - R_{k-1}) (\theta_0-\theta_k)
            \nn \\
            &=
                \theta_k +
                P_{k+1} \phi_k^\rmT (y_k -  \phi_k \theta_k) +
                P_{k+1} (R_k - R_{k-1}) (\theta_0-\theta_k).
            \nn 
    \end{align}
    Hence, \eqref{eq:theta_update_VDF} is satisfied. 
        %
    %
\end{proof}

Using $R_k-R_{k-1}= \Lambda_k A_{k-1} \Lambda_k - A_{k-1} $, it follows that \eqref{eq:theta_update_VDF} can be implemented without computing $P_k\inv$.

\end{document}